\documentclass[11pt, a4paper, reqno]{amsart}
\usepackage{multicol,a4wide}
\usepackage{amsmath}
\usepackage{amsaddr}
\usepackage{color}
\usepackage{cmap,mathtools}
\usepackage{cite}
\usepackage{xcolor}
\usepackage{hyperref}
\hypersetup{
    colorlinks=true,
    linkcolor=blue,
    filecolor=magenta,      
    urlcolor=green,
    citecolor=forestgreen,
}
\definecolor{forestgreen}{rgb}{0.13, 0.85, 0.15}
\usepackage[hyperpageref]{backref}
\usepackage{orcidlink}

\newcommand{\dx}{{\mathrm{d}x}}

\newcommand{\dy}{{\mathrm{d}y}}
\newcommand{\dxy}{\frac{\mathrm{d}x}{\abs{x}^a}\frac{\mathrm{d}y}{\abs{y}^a}}
\newcommand{\D}{{\mathcal{D}}}
\makeatletter
\newcommand{\leqnomode}{\tagsleft@true\let\veqno\@@leqno}
\usepackage{yhmath,amsmath,amsfonts,amssymb,enumerate,amsthm,appendix,caption,lipsum,}
\usepackage[T1]{fontenc}
\usepackage{enumitem}
\usepackage{mathtools,bm}
\usepackage{cmap}
\usepackage[hyperpageref]{backref}
\usepackage{bigints}
\usepackage{esint}
\usepackage[margin=0.90in]{geometry}

\newtheorem{definition}{Definition}[section]
\newtheorem{theorem}[definition]{Theorem}
\newtheorem{example}[definition]{Example}

\newtheorem{remark}[definition]{Remark}
\newtheorem{proposition}[definition]{Proposition}
\newtheorem{lemma}[definition]{Lemma}

\numberwithin{equation}{section}
\DeclarePairedDelimiter\abs{\lvert}{\rvert}
\DeclarePairedDelimiter\norm{\lVert}{\rVert}
\makeatletter
\let\oldnorm\norm
\def\norm{\@ifstar{\oldnorm}{\oldnorm*}}

\makeatletter
\newcommand{\al} {\alpha}
\newcommand{\pa} {\partial}

\newcommand{\la} {\lambda}

\newcommand{\Gr} {\nabla}
\newcommand{\no} {\nonumber}
\newcommand{\noi} {\noindent}

\newcommand{\ra} {\rightarrow}

\DeclareMathAlphabet{\mathpzc}{T1}{pzc}{m}{it}

\def\ps{p_{s}^{*}}
\def\psig{p_{\sigma}^{*}}

\def\Wsp{{\mathring{W}^{s,p}_a(\R^N)}}

\def\Dsp{{{\mathcal D}^{s,p}_a(\R^N)}}

\def\R{{\mathbb R}}
\def\N{{\mathbb N}}

\def\ep{{\epsilon}}
\def\dx{{\rm d}x}

\def\dz{{\rm d}z}
\def\dt{{\rm d}t}

\def\cc{{\mathcal{C}_c^{\infty}}}
\def\C{{\mathcal{C}}}
\def\del{{\partial}}

\def\Lpastar{{L^{\ps}_{a}(\R^N)}}

\def\Lasig{{L^{\psig}_{a}(\R^N)}}
\def\delpas{{\left(-\Delta\right)_{p,a}^s}}

\title[Higher order fractional weighted homogeneous spaces]{Higher order fractional weighted homogeneous spaces: characterization and finer embeddings}
\author[N. Biswas and R. Kumar]{Nirjan Biswas$^1$\orcidlink{0000-0002-3528-8388} \and Rohit Kumar$^2$\orcidlink{0009-0001-6494-6407}} 
\email{nirjaniitm@gmail.com, nirjan.biswas@acads.iiserpune.ac.in, kumar.174@iitj.ac.in}  
\thanks{$^1$Corresponding author}

\begin{document}
\vspace*{-5em}
\noindent\rule{\textwidth}{0.2pt} % Top horizontal line
\vspace{-1.5em}
\begin{center}
    \textbf{DOI:} \href{https://doi.org/10.1016/j.jmaa.2024.128935}{https://doi.org/10.1016/j.jmaa.2024.128935}\\ % Replace with your actual DOI
    \textbf{Published in:} Journal of Mathematical Analysis and Applications (2024)\\[-1ex] % Optional, add if relevant
\end{center}
\vspace{-1em}
\noindent\rule{\textwidth}{0.2pt} % Bottom horizontal line
\vspace{2em}
\maketitle
\centerline{$^1$Department of Mathematics, Indian Institute of Science Education and Research Pune,}
\centerline{Dr. Homi Bhabha Road, Pune 411008, India}
\centerline{$^2$ Department of Mathematics, Indian Institute of Technology Jodhpur,}
\centerline{Rajasthan 342030, India}

\begin{abstract}
In this article, for $N \geq 2, s \in (1,2), p\in (1, \frac{N}{s}), \sigma=s-1 $ and $a \in [0, \frac{N-sp}{2})$, we establish an isometric isomorphism between the higher order fractional weighted Beppo-Levi space
\begin{align*}
    {\mathcal D}^{s,p}_a(\mathbb{R}^N) := \overline{\cc(\mathbb{R}^N)}^{[\cdot]_{s,p,a}} \text{ where } [u]_{s,p,a} := \left( \iint_{\mathbb{R}^N \times \mathbb{R}^N} \frac{\left| \nabla u(x) -\nabla u (y) \right|^p}{\left|x-y \right|^{N+\sigma p}}\, \dxy \right)^{\frac{1}{p}},
\end{align*}
and higher order fractional weighted homogeneous space
\begin{align*}
    \mathring{W}^{s,p}_a(\mathbb{R}^N):= \left\{u \in L_a^{p^*_s}(\mathbb{R}^N): \norm{\Gr u}_{L_a^{p^*_{\sigma}}(\mathbb{R}^N)} + [u]_{s,p,a} < \infty \right\} 
\end{align*}
with the weighted Lebesgue norm
\begin{align*}
    \norm{u}_{L_a^{p^*_{\alpha}}(\mathbb{R}^N)}:=  \left( \int_{\mathbb{R}^N} \frac{ \abs{u(x)}^{p^*_{\alpha}}}{\abs{x}^{\frac{2ap^*_{\alpha}}{p}}} \, \dx \right)^{\frac{1}{p^*_{\alpha}}}, \text{ where } p^*_{\alpha}=\frac{Np}{N-\alpha p} \text{ for } \alpha= s,\sigma. 
\end{align*}
To achieve this, we prove that $\mathcal{C}_c^{\infty}(\mathbb{R}^N)$ is dense in $\mathring{W}^{s,p}_a(\mathbb{R}^N)$ with respect to $[\cdot]_{s,p,a}$, and $[\cdot]_{s,p,a}$ is an equivalent norm on $\mathring{W}^{s,p}_a(\mathbb{R}^N)$. Further, we obtain a finer embedding of ${\mathcal D}^{s,p}_a(\mathbb{R}^N)$ into the Lorentz space $L^{\frac{Np}{N-sp-2a}, p}(\mathbb{R}^N)$, where $L^{\frac{Np}{N-sp-2a}, p}(\mathbb{R}^N) \subsetneq L_a^{p^*_s}(\mathbb{R}^N)$. 
\end{abstract} 
\maketitle

\section{Introduction}
\noindent 
For $a \geq 0, s=1+\sigma, \sigma \in (0,1), p \in (1, \infty)$ and $N \ge 2$, we define the higher order fractional weighted-$p$ Laplace operator of order $s$, up to a normalizing constant as 
\begin{align*}
    &\langle \delpas u, \varphi \rangle = \iint_{\R^N \times \R^N} \frac{\left| \Gr u(x)-\Gr u(y)\right|^{p-2}\left(\Gr u(x)-\Gr u(y) \right) \cdot \left(\Gr \varphi(x)-\Gr \varphi(y)\right)}{\abs{x-y}^{N+\sigma p}}\,\dxy,
\end{align*}
for $u, \varphi \in \cc(\R^N)$. We define the following Gagliardo seminorm 
\begin{align*}
    [u]_{s,p,a} := \left( \iint_{\R^N \times \R^N} \frac{\left|\nabla u(x) -\nabla u (y) \right|^p}{\left|x-y \right|^{N+\sigma p}}\, \dxy \right)^{\frac{1}{p}}.
\end{align*}
For $a>0$, we consider the higher order fractional weighted Sobolev space $W_a^{s,p}(\R^N)$ as
\begin{align}
   W_a^{s,p}(\R^N):= \left\{ u \text{ is measurable} : \norm{u}_{p,a} + \norm{\Gr u}_{p,a} + [u]_{s,p,a} < \infty \right\},
\end{align}
where
\begin{align*}
   \norm{u}_{p,a}:= \left(\int_{\R^N} \abs{u(x)}^p \frac{\dx}{\abs{x}^a} \right)^{\frac{1}{p}} \text{ and } \norm{\Gr u}_{p,a}:= \left(\int_{\R^N} \abs{\Gr u(x)}^p \frac{\dx}{\abs{x}^a} \right)^{\frac{1}{p}}.
\end{align*}
For $a=0$, the corresponding higher order fractional Sobolev space $W_0^{s,p}(\R^N)$ is defined in \cite[Chapter 11]{Giovanni_book}. We refer to \cite{Giovanni_book} for more details on this space. Now we consider the following non-local elliptic problem
\begin{align}\label{nonlocal problem}
    \delpas u= f, \; \text{ in } \R^N,
\end{align}
under suitable assumptions on the source term $f$. For the existence of a weak solution to \eqref{nonlocal problem}, an obvious choice is to apply the Direct Method in the Calculus of Variations. Consequently, the problem reduces to minimizing the associated energy functional $\mathcal{J}$ given by
\begin{align}\label{functional}
    \mathcal{J}(u)= \frac{1}{p}\iint_{\R^N \times \R^N} \frac{\left|\nabla u(x) -\nabla u (y) \right|^p}{\left|x-y \right|^{N+\sigma p}}\, \dxy - \int_{\R^N} f(x)u(x) \, \dx.
\end{align}
Suppose, we consider the solution space for \eqref{nonlocal problem} as $W_a^{s,p}(\R^N)$ where $f \not\equiv 0$ and $f \in (W_a^{s,p}(\R^N))^*$ (dual space of $W_a^{s,p}(\R^N)$). Observe that one can not have any control over $\norm{\cdot}_{p, a}$ of minimizing sequences because the following weighted Poincar\'{e} type inequality
\begin{align}\label{Poincare}
    \int_{\R^N} \abs{u(x)}^p \, \frac{\dx}{\abs{x}^a} \leq C [u]_{s,p,a}^p, \; \forall \, u \in W_a^{s,p}(\R^N),
\end{align}
does not hold for any $C>0$. Suppose \eqref{Poincare} holds for some $C_1>0$. For $\la>0$ taking $u_{\la}(x) = u(\la x)$, we get
\begin{align*}
    [u_{\la}]_{s,p,a}^p = \la^{2a-N+(\sigma+1)p} [u]_{s,p,a}^p \text{ and } \norm{ u_{\la} }_{p,a}^p = \la^{a-N} \norm{u}^p_{p,a},    
\end{align*}
which in turn implies 
\begin{align}\label{poin1}
    \int_{\R^N} \abs{u(x)}^p \, \frac{\dx}{\abs{x}^a} \leq C_1 \la^{a+(\sigma+1)p} [u]_{s,p,a}^p, \text{ for every } \la >0.
\end{align}
Dividing \eqref{poin1} by $\la^{a+(\sigma+1)p}$ and taking $\la \ra 0$ we get $[u]_{s,p,a}= \infty$. In a similar way, one can verify that the following inequality
\begin{align*}
    \int_{\R^N} \frac{\abs{\Gr u(x)}^p}{\abs{x}^a} \, \dx \le C [u]_{s,p,a}^p, \; \forall \, u \in W_a^{s,p}(\R^N),
\end{align*}
does not hold for any $C>0$. This phenomenon conveys the fact that $W_a^{s,p}(\R^N)$ is not an appropriate function space to look for the weak solution for \eqref{nonlocal problem} as $\mathcal{J}$ is not coercive on $W_a^{s,p}(\R^N)$. 

The natural spaces to look for the existence of weak solutions for \eqref{nonlocal problem} are the higher order fractional weighted Beppo-Levi spaces $\Dsp$, which are defined by
\begin{align*}
    \Dsp:= \text{closure of } \cc(\R^N) \text{ with respect to } [\cdot]_{s,p,a}.
\end{align*}
The notation $\Dsp$ is reminiscent of the historical one introduced by Deny and Lions in their seminal paper \cite{Deny_Lions_1954}, where the authors first studied homogeneous spaces obtained by the completion of $\cc(\R^N)$ with respect to a given seminorm. Notice that for $f \in (\Dsp)^*$, $\mathcal{J}$ is well defined and  
\begin{align*}
    \mathcal{J}(u) \ge \frac{[u]_{s,p,a}^p}{p} - \norm{f}_{(\Dsp)^*} [u]_{s,p,a}, \; \forall\, u \in \Dsp. 
\end{align*}
The above inequality conveys that $\mathcal{J}$ is coercive on $\Dsp$. Hence, the existence of a solution for \eqref{nonlocal problem} can be assured in $\Dsp$ if one verifies other qualitative properties of $\mathcal{J}$; for example, the weak lower semicontinuity of $\mathcal{J}$ in $\Dsp$. The main focus of this article is to characterize $\Dsp$ as a function space. Verifying that a function $u \in \Dsp$ can be challenging, as it is not always feasible to find a sequence $\{u_n\} \subset \cc(\R^N)$ that converges to $u$ with respect to $[\cdot]_{s,p, a}$. To address this difficulty, we study an equivalent characterization of $\Dsp$. For $s=1+\sigma, \sigma \in (0,1), a \in [0, \frac{N-sp}{2})$ and $p\in (1, \frac{N}{s})$, we define the higher order fractional weighted homogeneous space $\Wsp$ as 
\begin{align}\label{Wsp space}
    \mathring{W}^{s,p}_a(\R^N):= \left\{u \in L_a^{p^*_s}(\R^N) : \norm{u}_{\Wsp} < \infty \right\},
\end{align}
where
\begin{align*}
    \norm{u}_{\Wsp} := \norm{\Gr u}_{L_a^{p^*_{\sigma}}(\R^N)} + [u]_{s,p,a}, 
\end{align*}
where the weighted Lebesgue space $L^{p^*_{\alpha}}_a(\R^N)$, for $\alpha=s,\sigma$, is given by
\begin{align*}
L_a^{p^*_{\al}}(\R^N):= \left\{ u \text{ is measurable} : \norm{u}_{L_a^{p^*_{\al}}(\R^N)} < \infty  \right\}, \text{ where } \norm{u}^{p^*_{\al}}_{L_a^{p^*_{\al}}(\R^N)}  := \int_{\R^N} \frac{\abs{u(x)}^{p^*_{\al}}}{\abs{x}^{\frac{2ap^*_{\al}}{p}}} \, \dx,
\end{align*}
and the fractional critical exponent $p^*_{\al}$ is defined as $p^*_{\al} :=\frac{Np}{N-\alpha p}$. The name `\textit{homogeneous}' is due to the fact that 
\begin{align*}
    \norm{u_{\la}}_{\mathring{W}^{s,p}_a(\R^N)} = \norm{u}_{\mathring{W}^{s,p}_a(\R^N)}, \text{ for } u_{\la}(x) =  \la^{\frac{N-sp-2a}{p}} u(\la x), \text{ where } \la>0. 
\end{align*}
In the first part of this article, we aim to establish the existence of an isometric isomorphism between $\Wsp$ and $\Dsp$.
For that we investigate certain qualitative properties of $\Wsp$, such as the density of $\cc(\R^N)$ in $\Wsp$ and the equivalence between $[\cdot]_{s,p,a}$ and $\norm{\cdot}_{\Wsp}$.

For $s\in (0,1]$, a similar characterization for the Beppo Levi space is recently studied in \cite{Brasco_Characterisation}.
For $a=0, s\in (0,1]$ and $p \in (1, \frac{N}{s})$, the Beppo-Levi space $\mathcal{D}^{s,p}(\R^N)$ is defined as the completion of $\cc(\R^N)$ with respect to the seminorms 
\begin{align*}
    \norm{u}_{s,p}:= \begin{cases}
        \displaystyle \left(\iint_{\R^N \times \R^N}\frac{\abs{u(x)-u(y)}^p}{\abs{x-y}^{N+sp}}\,\dx \dy\right)^{\frac{1}{p}},& \text{ for } s\in (0,1),\\
        \displaystyle \left(\int_{\R^N} \abs{\Gr u}^p\,\dx\right)^{\frac{1}{p}},& \text{ for } s=1.
    \end{cases}
\end{align*}
From \cite[Theorem 10.2.1]{Mazya_Sobolev_Spaces} and \cite[Theorem 9.9]{Brezis_book}, the following Sobolev inequality holds
\begin{align}\label{Sobolev inequality}
    \norm{u}_{L^{\ps}(\R^N)} \leq C(N,p,s)\norm{u}_{s,p}, \; \forall \, u \in \cc(\R^N).
\end{align}
In \cite[Theorem 3.1]{Brasco_Characterisation}, using \eqref{Sobolev inequality}, Brasco et al. characterize the fractional Beppo-Levi space $\mathcal{D}^{s,p}(\R^N)$ as 
\begin{align*}
    \mathcal{D}^{s,p}(\R^N) = \left\{u \in L^{\ps}(\R^N) : \norm{u}_{s,p} <\infty \right\}.
\end{align*}
For $a \in (0, \frac{N-sp}{2})$ the following weighted Sobolev inequality holds (see \cite[Theorem 1.5]{Caffarelli_Kohn_2017}):
\begin{align}\label{weighted fractional}
 \norm{u}_{\Lpastar} \leq C(N,p,s,a) \norm{u}_{s,p,a}, \; \forall \, u \in \cc(\R^N),
\end{align}
where $\norm{\cdot}_{s,p,a}$ is defined as 
\begin{align*}
    \norm{u}_{s,p,a}:=  \left(\iint_{\R^N \times \R^N} \frac{\abs{u(x)-u(y)}^p}{\abs{x-y}^{N+sp}}\,\dxy \right)^{\frac{1}{p}}.
\end{align*}
Motivated by \eqref{weighted fractional}, the authors in \cite{Density_property} define $\Wsp$ endowed with the following norm 
$$\norm{u}_{\mathring{W}_a^{s,p}(\R^N)}:= \norm{u}_{s,p,a} + \norm{u}_{\Lpastar}.$$
In \cite[Theorem 1.1]{Density_property}, the authors proved that $\cc(\R^N)$ is dense in $\Wsp$. Clearly, this density result and \eqref{weighted fractional} infer that $\norm{\cdot}_{s,p,a}$ is an equivalent norm in $\Wsp$. The appearance of the space $\Wsp$ is seen in some literature; for example, Hardy-Lieb-Thirring inequalities (see \cite{Frank_Lieb_2008}) and the non-local
critical problems involving the Hardy-Leray potential (see \cite{Dipierro-Montoro-2015}), for the normalized solutions of some fractional Schr\"{o}dinger equations (see \cite{Xiang2024}).

For  $s \in (0,1)$, $a\in [0,\frac{N-sp}{2})$ and $p\in (1,\frac{N}{s})$, from \cite[Theorem 2.7]{Caffarelli_Kohn_2017} we now recall the following weighted Hardy inequality 
\begin{align}\label{weighted Hardy_intro}
  \int_{\R^N}\frac{\abs{u(x)}^p}{\abs{x}^{sp +2a}}\,\dx \le C\iint_{\R^N \times \R^N} \frac{\abs{u(x)-u(y)}^p}{\abs{x-y}^{N + sp}}\, \dxy, \; \forall \, u \in \cc(\R^N),
\end{align} 
where $C=C(N,p,s,a)>0$. For $a=0$, \eqref{weighted Hardy_intro} reduces to the fractional Hardy inequality derived by Frank-Seiringer in \cite{Frank-JFA-2008} with the best constant. Further, using this inequality the authors in \cite[Lemma 4.3]{Frank-JFA-2008} established an embedding of $\D^{s,p}(\R^N)$ into the Lorentz space $L^{p^*_s, p}(\R^N)$, namely $\norm{u}_{p^*_s, p} \le C(N,p,s) [u]_{s,p}$, for all $u \in \Dsp$. This embedding is finer than the fractional Sobolev embedding $\mathcal{D}^{s,p}(\R^N) \hookrightarrow L^{p^*_s}(\R^N)$ due to the fact that $L^{p^*_s, p}(\R^N) \subsetneq L^{p^*_s, p^*_s}(\R^N) = L^{p^*_s}(\R^N)$. In the local case $s=1$, the Lorentz Sobolev embeddings are studied for $W^{1,p}(\R^N)$ and $\D^{1,p}(\R^N)$. For example, the authors in \cite{AN2019, Peetre-1966, Tartar1998} independently demonstrated that $W^{1,p}(\R^N) \hookrightarrow L^{p^*,p}(\R^N)$ where $p^*= \frac{Np}{N-p}$ for $N>p$. In \cite{ANU2021}, the authors showed that $\mathcal{D}^{1,p}(\R^N) \hookrightarrow L^{p^*,p}(\R^N)$. For $s \in (1,2)$, the corresponding fractional Lorentz Sobolev embedding is not known. However, for $s=2$, in \cite{MaEv} the authors showed that the Beppo Levi space $\mathcal{D}^{2,2}(\R^N)$ (which is a completion of $\cc(\R^N)$ with respect to $\norm{\Delta u}_{L^2(\R^N)}$) is continuously embedded into the Lorentz space $L^{2^{**}, 2}(\R^N)$ where $2^{**} = \frac{2N}{N-4}$.

In the second part of this article, for $s \in (1,2)$, we establish an embedding $\Dsp$ into the Lorentz space $L^{p^{\star}, p}(\R^N)$, where $p^{\star} = \frac{Np}{N-sp-2a}$. This embedding is finer since $L^{p^{\star}, p}(\R^N) \subsetneq L^{p^*_s}_a(\R^N)$ (see Proposition \ref{finer}). Moreover, $p^{\star}$ is an appropriate exponent in this case, since for $s=p=2$ and $a=0$, $p^{\star}$ becomes $\frac{2N}{N-4}$. 
For this finer embedding, we derive a fractional weighted Rellich inequality in Proposition \ref{weighted higher}. For $a=0$ and $p=2$, a similar inequality has been obtained in \cite[Corollary 7.5]{EdEv}.

The following theorem collects the main results of this article. 

\begin{theorem}\label{Theorem1.1}
Let $N \ge 2, s \in (1,2), \sigma = s-1, p \in (1, \frac{N}{s})$, and $a \in [0, \frac{N-sp}{2})$. Then the following hold: 
\begin{enumerate}
    \item[\rm{(i)}](\textbf{Density of smooth functions}) For every $u \in \Wsp$, there exists $u_n \in \cc(\R^N)$ such that
    \begin{align*}
        [u-u_n]_{s,p,a} \ra 0, \text{ as } n \ra \infty, 
    \end{align*}
    namely $\cc(\R^N)$ is dense in $\Wsp$ with respect to the seminorm $[\cdot]_{s,p,a}$. 
    \item[\rm{(ii)}](\textbf{Equivalent norm}) There exists $C=C(N,p, \sigma,a)>0$ such that 
    \begin{align*}
        \norm{\Gr u}_{L_a^{p^*_{\sigma}}(\R^N)} \leq C \left[u\right]_{s,p,a}, \; \forall \, u \in \Wsp.  
    \end{align*}
    \item[\rm{(iii)}](\textbf{Characterization}) $\Wsp$ is a Banach space with respect to the norm $[\cdot]_{s,p,a}$. Further, there exists an isometric isomorphism
\begin{align*}
    \mathcal{J}: \Dsp \ra \Wsp,
\end{align*}
i.e., the space $\Wsp$ is identified with $\Dsp$.
    \item[\rm{(iv)}](\textbf{Finer embeddings}) Denote $p^{\star} = \frac{Np}{N-sp-2a}$. Then 
    \begin{align}\label{finer embedding}
    |u|_{p^{\star}, p} \le C(N,s,p,a) [u]_{s,p,a}, \; \forall \, u \in \Dsp,
   \end{align}
i.e., $\Dsp \hookrightarrow L^{p^{\star}, p}(\R^N)$.
\end{enumerate}
\end{theorem}

The rest of the article is organized as follows. In Section 2, we state mollifier and cut-off functions and define Lorentz spaces, weak Young's inequality. In Section 3, we prove the approximations with smooth functions and norm equivalence in $\Wsp$. Section 4 is devoted to study the characterization between $\Wsp$ and $\Dsp$. In the final section, we prove fractional weighted Rellich inequality and finer embeddings of $\Dsp$.

\section{Preliminaries}

We begin by proving several basic inequalities to be used frequently in the subsequent sections.
Next, we introduce the mollifier and cut-off functions essential for establishing the density results. Subsequently, we define the Lorentz spaces and state the weak Young's inequality. The following notations will be used throughout this paper.

\noi \textbf{Notation:}
\begin{enumerate}
    \item[(i)]  For $p \in [1,\infty]$, $L^p(\R^N)$-norm of $f$ is denoted as $\norm{f}_p$.
    \item[(ii)] $B_r:=B_r(0) \subset \R^N$ denotes a ball of radius $r$ with centre at the origin.
    \item[(iii)] The convolution $f \star g$ is defined as $(f \star g)(x) = \int_{\R^N} f(x-z) g(z) \, \dz.$
    \item[(iv)] For $\rho_n$ defined in \eqref{molli}, we denote 
    \begin{align*}
        u_n = u \star \rho_n \text{ and } \frac{\partial u_n}{\partial x_i} = \frac{\partial u}{\partial x_i} \star \rho_n,
    \end{align*}
    for $i=1,\cdot,\cdot,\cdot,N$. 
    \item[(v)]For $A \subset \R^N$, it's complement $\R^N \setminus A$ is denoted as $A^c$.
    \item[(vi)] $C$ denotes generic positive constant. 
\end{enumerate}

\begin{lemma}\label{lemma_1} Let $a_i \geq 0$, for $i=1,2,\cdot,\cdot,\cdot,N$ and $0<q<\infty$. Then
  \begin{itemize}
      \item[\rm{(i)}] $\left(a_1 + a_2+ ...+a_N \right)^q \leq A(N,q) \left(a_1^q + a_2^q +...+a_N^q \right)$,
      \item[\rm{(ii)}]  $\left(a_1 + a_2+ ...+a_N \right)^q \geq B(N,q) \left(a_1^q + a_2^q +...+a_N^q \right)$, 
  \end{itemize} 
where 
\begin{align*}
    A(N,q)=1,~ B(N,q) =\begin{cases}
     \left(\frac{q}{2}\right)^{\frac{N+1}{2}},& \text{ if } N \text{ is odd }\\
     \left(\frac{q}{2}\right)^{\frac{N}{2}},& \text{ if } N \text{ is even}
 \end{cases}, \text{ for } 0<q<1,
\end{align*}
and 
\begin{align*}
    B(N,q)=1,~ A(N,q) =\begin{cases}
          \left(2^{q-1}\right)^{\frac{N+1}{2}}, & \text{ if } N \text{ is odd}\\
          \left(2^{q-1}\right)^{\frac{N}{2}},& \text{ if } N \text{ is even}
      \end{cases}, \text{ for } 1\leq q < \infty.
\end{align*}
\end{lemma}
\begin{proof}
 $\rm{(i)}$ \textbf{For \boldmath$0<q<1$:} In this case, $(a_1 + a_2)^q \leq a_1^q + a_2^q$ and generalising this inequality we get
 \begin{align*}
     \left(a_1 + a_2+ ... +a_N \right)^q \leq \left(a_1^q + a_2^q +...+a_N^q \right).
 \end{align*}
\noindent \textbf{For \boldmath$1\leq q<\infty$:} In this case, using the convexity of $x \mapsto |x|^q$, we have $(a_1 + a_2)^q \leq 2^{q-1} (a_1^q + a_2^q)$ for any $a_1, a_2 \ge 0$. Using it, we find
 \begin{align*}
    (a_1 + a_2 + a_3)^q \leq 2^{q-1}(a_1 + a_2)^q + 2^{q-1} a_3^q ~\leq~ \left(2^{q-1}\right)^2(a_1^q + a_2^q + a_3^q),
 \end{align*} and 
 \begin{align*}
     (a_1 + a_2 + a_3+ a_4)^q \leq 2^{q-1}(a_1 + a_2)^q + 2^{q-1}(a_3+ a_4)^q  \leq \left(2^{q-1}\right)^2(a_1^q + a_2^q + a_3^q+a_4^q).
 \end{align*}
Hence, upon generalization, 
\begin{align*}
    \left(a_1 + a_2+ ...+a_N \right)^q \leq A(N,q) \left(a_1^q + a_2^q +...+a_N^q \right).
\end{align*}
\noindent $\rm{(ii)}$ \textbf{For \boldmath$0<q<1$:} We first prove that 
    \begin{equation} \label{a1}
        (a_1 + a_2)^q \geq \frac{q}{2}(a_1^q + a_2^q).
    \end{equation}
We define $f_1: [1,\infty) \rightarrow \R$ by $f_1(t) = (t+1)^q - \frac{q}{2} (t^q +1)$. Then $f_1(t)\geq \left(1-\frac{q}{2} \right)t^q -\frac{q}{2} \geq 1-\frac{q}{2}  -\frac{q}{2}>0,$ which implies 
\begin{equation}\label{a2}
    \left( t+1 \right)^q > \frac{q}{2} (t^q+1), \text{ for all } t \geq 1.
\end{equation}
Now \eqref{a1} is trivial if either $a_1=0$ or $a_2=0$. So, we assume $a_1,a_2 \neq 0$. If $a_1 \geq a_2$, then for $t =\frac{a_1}{a_2}$ in \eqref{a2} we get \eqref{a1}. If $a_1\leq a_2$, then for $t = \frac{a_2}{a_1}$ in \eqref{a2} we get \eqref{a1}. 
From \eqref{a1} we deduce that
\begin{align*}
    (a_1 + a_2 + a_3)^q \geq \frac{q}{2}(a_1 + a_2)^q + \frac{q}{2} a_3^q \geq \left(\frac{q}{2}\right)^2(a_1^q + a_2^q + a_3^q),
 \end{align*} and 
 \begin{align*}
  (a_1 + a_2 + a_3+ a_4)^q \geq \frac{q}{2}(a_1 + a_2)^q + \frac{q}{2}(a_3+ a_4)^q  \geq \left(\frac{q}{2}\right)^2(a_1^q + a_2^q + a_3^q+a_4^q).   
 \end{align*}
Upon generalization, we obtain
\begin{align*}
    \left(a_1 + a_2+ ...+a_N \right)^q \geq B(N,q) \left(a_1^q + a_2^q +...+a_N^q \right).
\end{align*}
\textbf{For \boldmath$1\leq q<\infty$:} We consider a function $f_2 : [1,\infty) \rightarrow \R$ by $f_2(t)= (t+1)^q-t^q-1$.
Then $f_2'(t) = q[(t+1)^{q-1}-t^{q-1}] \geq 0$. Thus, $f(t) \geq f(1)$ for all $t \geq 1$, which implies $(t+1)^q-t^q-1 \geq 2^q-2 \geq0$ as $q\geq 1$. Thus,
\begin{equation}\label{a3}
    (t+1)^q \geq t^q+1, \text{ for all } t \geq 1.
\end{equation}
By substituting $t=\frac{a_1}{a_2}$ if $a_1 \geq a_2$ or $t=\frac{a_2}{a_1}$ if $a_1 \leq a_2$ into \eqref{a3}, we get
\begin{equation}\label{a4}
    (a_1 + a_2)^q \geq a_1^q + a_2^q.
\end{equation}
Generalizing \eqref{a4}, $(a_1 + a_2+...+a_N)^q \geq a_1^q + a_2^q+...+a_N^q$.
\end{proof}

\begin{definition}[Mollifier functions]\label{Mollifier functions}
 Consider a radially symmetric and decreasing function $\rho \in \cc(\R^N)$ such that $\rho \geq 0$ on $\R^N$, $\text{supp}(\rho) \subset \overline{B_1}$ and $\int_{\R^N} \rho(x)\,dx=1$. For each $n \in \N$, define 
\begin{equation}\label{molli}
    \rho_n(x) = n^N \rho(nx), \text{ for } x \in \R^N.
\end{equation}
Then $\rho_n$ satisfies 
\begin{align*}
    \rho_n \in \cc(\R^N),\; \text{supp}(\rho_n) \subset \overline{B_{\frac{1}{n}}}, \; \int_{\R^N} \rho_n(x)\,dx =1, \; \rho_n \geq 0 \text{ in } \R^N.
\end{align*}
The sequence $\{\rho_n\}$ is called a sequence of mollifiers.   
\end{definition}

\begin{definition}[Cut-off functions]
  Fix a function $\zeta \in \cc(\R^N)$ such that $0\leq \zeta \leq 1$ and 
\begin{align*}
    \zeta(x) = \begin{cases}
    1 &\quad \text{ if } |x|<1,\\
    0 &\quad  \text{ if } |x|\geq 2.
    \end{cases}
\end{align*}
For each $n \in \N$, define 
\begin{equation}\label{cutoff}
    \zeta_n(x):= \zeta\left(\frac{x}{n}\right).
\end{equation}
The sequence $\{\zeta_n\}$ is called the cut-off function. Observe that
\begin{align}\label{estimate zeta_n}
    \|\nabla \zeta_n \|_{\infty}\leq \frac{C}{n},  \text{ and } \|\nabla^2 \zeta_n\|_{\infty} \leq \frac{C}{n^2},
\end{align}
where $C=C(N)$ and $\nabla^2$ represents the Hessian matrix.  
\end{definition}

Now, we recall the definition of Lorentz space introduced by Lorentz in \cite{Lorentz}. For an open set $\Omega \subset \R^N$, let $M(\Omega)$ be the set of all extended real-valued Lebesgue measurable functions that are finite a.e. in $\Omega$. For $t > 0$ and $f \in M(\Omega)$, define $E_{f}(t) := \{ x \in \Omega : \abs{f(x)} > t \}$. The distribution function $\mu_{f}$ of $f$ is defined as $\mu_f(t) := \abs{E_{f}(t)}$. For $(p,q) \in (0,\infty) \times (0,\infty]$ we consider the following quantity
 \begin{align*}
 |f|_{p,q} := \left\{\begin{array}{ll}
 \left( p \displaystyle \int_0^\infty t^{q-1} {\mu_f(t)}^{\frac{q}{p}} \, \dt \right)^{\frac{1}{q}}, \quad & q < \infty; \\ 
 \displaystyle \sup_{t>0} t\mu_f(t)^{\frac{1}{p}},\quad & q=\infty.
 \end{array} 
 \right. 
 \end{align*}
The Lorentz space $L^{p,q}(\Omega)$ is defined as
 \[ L^{p,q}(\Omega) := \left \{ f \in M(\Omega): \, |f|_{p,q} < \infty \right \},\]
where $ |f|_{p,q}$ is a quasi norm on $L^{p,q}(\Omega).$ Notice that $L^{p,p}(\R^N) = L^p(\R^N)$ and $L^{p,q_1}(\R^N) \subsetneq L^{p, q_2}(\R^N)$ provided $q_1<q_2$ (see \cite[Proposition 3.4.3]{EdEv}). 
For details about this space, we refer to  \cite{EdEv}.

\begin{example}\label{example-weak-Lp}
   For $d<N$, consider the function $g(x) = \abs{x}^{-d}$ where $x \in \R^N$. Using \cite[Example 2.2.5]{biswas2021} we get $g \in L^{\frac{N}{d}, \infty}(\R^N)$.
\end{example}

In the following proposition, we state the weak Young's inequality. For more details and proofs of this inequality, we refer \cite{Lieb2001}. 

\begin{proposition}\label{convolution weak}
Let $1<p,q ,r < \infty$ be such that $1+\frac{1}{r}=\frac{1}{p} + \frac{1}{q}$. Let $f \in L^{p,\infty}(\R^N)$ and $h \in L^q(\R^N)$. Then $f \star h \in L^r(\R^N)$, and
\begin{align*}
    \norm{f \star h}_{r} \le C(N,p,q) \abs{f}_{p, \infty} \norm{h}_q.
\end{align*}
\end{proposition}

\section{Approximations by smooth functions with compact support}

This section is devoted to proving the density of smooth functions in $\Wsp$ and the equivalence of $[\cdot]_{s,p,a}$ with the $\Wsp$ norm. Given a measurable function $w:\R^{2N} \ra \R$, define 
\begin{align*}
    w * \rho_n(x,y) := \int_{\R^N}  w(x-z,y-z)\rho_n(z)\,\dz.
\end{align*}
Then the following hold.
\begin{proposition}\label{bound1}
(i) For every measurable function $w_1: \R^{2N} \ra \R$,
    \begin{align*}
    \iint_{\R^{N} \times \R^N} \left|w_1 * \rho_n(x,y) \right|^p\,\dxy \leq C \iint_{\R^{N}\times \R^N} \left|w_1(x,y) \right|^p\,\dxy,
    \end{align*}
    where $C$ is independent of $n$.

\noi (ii) Let $\alpha=s, \sigma$. For every measurable function $w_2 : \R^N \ra \R$,
\begin{align*}
    \int_{\R^N} \frac{|w_2 \star \rho_n(x)|^{p^*_{\alpha}}}{\abs{x}^{\frac{2ap^*_{\alpha}}{p}}} \, \dx \le C \int_{\R^N} \frac{|w_2(x)|^{p^*_{\alpha}}}{\abs{x}^{\frac{2ap^*_{\alpha}}{p}}} \, \dx,
\end{align*}
where $C$ is independent of $n$.
\end{proposition}

\begin{proof}
    Proof follows using \cite[Proposition 4.4 and Proposition 4.5]{Density_property} and the fact that for each $n \in \N$, $\text{supp}(\rho_n) \subset B_1$. 
\end{proof}
For $u \in \cc(\R^N)$, it follows from \cite[Lemma 2.1]{Density_property} with $\frac{\del u}{\del x_i}$ that $[u]_{s,p,a}<\infty$. Moreover, using $a \in [0, \frac{N-sp}{2})$ and $\sigma < s$,
\begin{align*}
    & \norm{u}_{\Lpastar}^{\ps}= \int_{\R^N} \frac{|u(x)|^{\ps}}{|x|^{\frac{2a \ps}{p}}}\,\dx \leq \norm{u}^{\ps}_\infty \int_{\text{supp}(u)} \frac{\dx}{|x|^{\frac{2a \ps}{p}}}<\infty, \\
    & \norm{\Gr u}_{\Lasig}^{\psig}= \int_{\R^N} \frac{|\Gr u(x)|^{\psig}}{|x|^{\frac{2a \psig}{p}}}\,\dx \leq \norm{\Gr u}^{\psig}_\infty \int_{\text{supp}(u)} \frac{\dx}{|x|^{\frac{2a \psig}{p}}}<\infty.
\end{align*}
Thus we observe that $\cc(\R^N) \subset \Wsp$.

\begin{proposition}\label{Prop3.2}
   Let $N \ge 2, s \in (1,2), \sigma = s-1, p \in (1, \frac{N}{s})$, and $a \in [0, \frac{N-sp}{2})$. Let $\rho_n$ and $\zeta_n$ be defined as in \eqref{molli} and \eqref{cutoff} respectively. Then for every $u \in  \Wsp$,  $$\left[u - \left(u \star \rho_n \right)\zeta_n\right]_{s,p,a} \rightarrow 0, \text{ as } n \ra \infty.$$
\end{proposition}
\begin{proof}
   For $u \in \Wsp$ and $\rho_n \in \cc(\R^N)$, we have $u \star \rho_n \in \C^\infty(\R^N)$
   and $\left(u \star \rho_n\right) \zeta_n \in \cc(\R^N)$ (see \cite[Proposition 4.20]{Brezis_book}).
   %For each $n \in \N$, set $v_n = \left(u \star \rho_n\right) \zeta_n$. 
   Using the triangle inequality, we write
  \begin{align}\label{eq1}
      \left[u-\left(u \star \rho_n\right) \zeta_n \right]_{s,p,a} 
      \leq \left[u-u \star \rho_n\right]_{s,p,a} + \left[u \star \rho_n-\left(u \star \rho_n\right) \zeta_n  \right]_{s,p,a}.
  \end{align}
We now estimate the first term of \eqref{eq1}. Since $\left|\Gr u\right| \in L^{\psig}_{a}(\R^N)$, interchange of differentiation and integration gives  $\nabla (u \star \rho_n) = \nabla u \star \rho_n \in (L^{\psig}_{a}(\R^N))^N$ (applying Proposition \ref{bound1}-(ii) with $\frac{\pa u}{ \pa x_i}$). Using Lemma \ref{lemma_1}-(i), we estimate
\begin{align}\label{eq2}
    \left[u- u \star \rho_n \right]_{s,p,a}^p
    &= \iint_{\R^N \times \R^N} \frac{\left|\left(\nabla u(x)-\nabla u(y) \right) -\left(\left(\nabla u \star \rho_n\right)(x)-\left(\nabla u \star \rho_n\right) (y) \right) \right|^p}{|x-y|^{N+\sigma p}}\,\dxy \no\\
    &= \iint_{\R^N \times \R^N} \frac{\left( \sum_{i=1}^{N} \left( \frac{\partial u}{\partial x_i}(x)- \frac{\partial u}{\partial x_i}(y) 
    -\left(\frac{\partial u_n}{\partial x_i}(x)-\frac{\partial u_n}{\partial x_i}(y) \right) \right)^2 \right)^{\frac{p}{2}}}{|x-y|^{N+\sigma p}}\,\dxy \no\\
    & \le C(N,p) \iint_{\R^N \times \R^N} \frac{\sum_{i=1}^{N} \left| \frac{\partial u}{\partial x_i}(x)- \frac{\partial u}{\partial x_i}(y) 
    -\left(\frac{\partial u_n}{\partial x_i}(x)-\frac{\partial u_n}{\partial x_i}(y) \right) \right|^p}{|x-y|^{N+\sigma p}}  \,\dxy \no \\
    & = C(N,p) \sum_{i=1}^{N} \iint_{\R^N \times \R^N} \frac{\left| \frac{\partial u}{\partial x_i}(x)- \frac{\partial u}{\partial x_i}(y) 
    -\left(\frac{\partial u_n}{\partial x_i}(x)-\frac{\partial u_n}{\partial x_i}(y) \right) \right|^p}{|x-y|^{N+\sigma p}}  \,\dxy.
\end{align}
Further, using Lemma \ref{lemma_1}-(ii), we find
\begin{align*}
    \left[u\right]_{s,p,a}^p =\iint_{\R^N \times \R^N} \frac{\left|\nabla u(x) -\nabla u (y) \right|^p}{\left|x-y \right|^{N+\sigma p}}\, \dxy \geq C(N,p) \sum_{i=1}^{N}\iint_{\R^N \times \R^N} \frac{\left|\frac{\partial u}{\partial x_i}(x) -\frac{\partial u}{\partial x_i} (y) \right|^p}{\left|x-y \right|^{N+\sigma p}}\, \dxy,
\end{align*}
which implies that for each $i = 1, \cdot,\cdot,\cdot, N$, 
\begin{align}\label{wi norm}
    \iint_{\R^N \times \R^N} \left|w_i(x,y)\right|^p\, \dxy < \infty,
\end{align}
where
\begin{align}\label{wi}
    w_i(x,y) = \frac{\frac{\partial u}{\partial x_i}(x) -\frac{\partial u}{\partial x_i} (y)}{\left|x-y \right|^{\frac{N}{p}+\sigma}}, \text{ for } x,y \in \R^N.
\end{align} 
For each $i = 1, \cdot,\cdot,\cdot, N$, now we apply \cite[Lemma 6.1]{Density_property} with $\frac{\partial u}{\partial x_i}$ to obtain
\begin{equation}\label{e3}
    \lim\limits_{n \rightarrow \infty} \iint_{\R^N \times \R^N} \frac{\left| \frac{\partial u}{\partial x_i}(x)- \frac{\partial u}{\partial x_i}(y) 
    -\left(\frac{\partial u_n}{\partial x_i}(x)-\frac{\partial u_n}{\partial x_i}(y) \right) \right|^p}{|x-y|^{N+\sigma p}}\,\dxy=0. 
\end{equation}
Combining \eqref{eq2} and \eqref{e3}, we deduce that 
\begin{equation}\label{eq1.1} 
    \lim\limits_{n \rightarrow \infty}\left[u-u \star \rho_n\right]_{s,p,a} =0.
\end{equation}
Next, we estimate the second term on the right-hand side of the inequality \eqref{eq1}. Define $\varphi_n=1- \zeta_n$. Using $\nabla \left( \left(u \star \rho_n\right) \varphi_n\right) = \left(\nabla u \star \rho_n\right) \varphi_n + \left(u \star \rho_n\right) \nabla\varphi_n$ and Lemma \ref{lemma_1}-(i), we derive the following estimate:
\begin{align}\label{eq4}
    &\left[\left(u \star \rho_n\right) \varphi_n\right]_{s,p,a}^p \no\\
    %&= \iint\limits_{\R^{N} \times \R^N} \frac{\Big| \left[\left(\nabla u \star \rho_n\right)\varphi_n\right](x)-\left[\left(\nabla u \star \rho_n\right)\varphi_n\right](y) + \left[\left( u \star \rho_n\right) \nabla \varphi_n\right](x)-\left[\left( u \star \rho_n\right) \nabla \varphi_n\right](y)   \Big|^p}{|x-y|^{N+\sigma p}}\,\dxy \no\\
    & \le C \sum_{i=1}^{N} \iint_{\R^N \times \R^N} \frac{\left| \frac{\partial u_n}{\partial x_i}(x) \varphi_n(x) - \frac{\partial u_n}{\partial x_i}(y) \varphi_n(y) + u_n(x) \frac{\partial \varphi_n}{\partial x_i}(x) - u_n(y) \frac{\partial \varphi_n}{\partial x_i}(y) \right|^p}{|x-y|^{N+\sigma p}}  \,\dxy \no \\
    &\le C \Bigg( \sum\limits_{i=1}^{N} \iint_{\R^{N} \times \R^N} \frac{\left| \frac{\partial u_n}{\partial x_i}(x)\varphi_n(x) - \frac{\partial u_n}{\partial x_i}(y)\varphi_n(y) \right|^p}{|x-y|^{N+\sigma p}} \, \dxy \no \\
    & \quad \quad \quad \quad + \sum\limits_{i=1}^{N} \iint_{\R^{N} \times \R^N} \frac{\left| u_n(x) \frac{\partial \varphi_n}{\partial x_i}(x)-u_n(y)\frac{\partial \varphi_n}{\partial x_i}(y)  \right|^p}{|x-y|^{N+\sigma p}} \, \dxy \Bigg),
\end{align}
where $C=C(N,p)$.  Further, using Proposition \ref{bound1}-(i) and \eqref{wi norm}, we have
\begin{align}\label{unifrom-bound 3.8}
   \left[\frac{\partial u_n}{\partial x_i} \right]_{\sigma,p,a}^p &= \iint_{\R^{N} \times \R^N} \frac{\left| \left(\frac{\partial u}{\partial x_i}\star \rho_n\right)(x) - \left(\frac{\partial u}{\partial x_i} \star \rho_n\right)(y) \right|^p }{|x-y|^{N+\sigma p}} \, \dxy \no \\
   & =\iint_{\R^{N} \times \R^N}  \left| \int_{\R^N} \frac{\left(\frac{\partial u}{\partial x_i}(x-z) -\frac{\partial u}{\partial x_i}(y-z)\right) \rho_n(z)\,\dz}{|(x-z)-(y-z)|^{\frac{N}{p}+\sigma}} \right|^p  \, \dxy \no\\
   &=\iint_{\R^{N} \times \R^N} \left| w_i * \rho_n(x,y)\right|^p \, \dxy \leq C \iint_{\R^{N} \times \R^N} \left| w_i(x,y)\right|^p \, \dxy  \le C,
\end{align}
where $C$ does not depend on $n$. 
Since $\varphi_n(x)=0=\varphi_n(y)$ for $x,y \in B_n$, we write
\begin{align}\label{integral 3.9}
    \iint_{\R^N \times \R^N} \frac{\left|\frac{\partial u_n}{\partial x_i}(x)\varphi_n(x)-\frac{\partial u_n}{\partial x_i}(y)\varphi_n(y)\right|^p}{\abs{x-y}^{N+\sigma p}}\dxy \leq C (J_{n,i} + K_{n,i}),
\end{align}
where
\begin{align*}
    J_{n,i} &:= \iint_{\R^N \times B_n^c} \frac{\left|\frac{\partial u_n}{\partial x_i}(x)-\frac{\partial u_n}{\partial x_i}(y)\right|^p}{\abs{x-y}^{N+\sigma p}} |\varphi_n(x)|^p\dxy \\
  \text{ and }  K_{n,i}&:=\iint_{\R^N \times B_n^c} \frac{\left|\varphi_n(x)-\varphi_n(y)\right|^p}{\abs{x-y}^{N+\sigma p}}\left|\frac{\partial u_n}{\partial x_i}(y)\right|^p\dxy.
\end{align*}
Further, we define
\begin{align*}
    D_{n,0}&:= \left\{(x,y) \in \R^N \times B_n^c \text{ s.t. } |x| \leq \frac{\abs{y}}{2} \right\},\\
    D_{n,1}&:= \left\{(x,y) \in \R^N \times B_n^c \text{ s.t. } |x| > \frac{\abs{y}}{2} \text{ and } |x-y| \geq n \right\},\\
    D_{n,2}&:= \left\{(x,y) \in \R^N \times B_n^c \text{ s.t. } |x| > \frac{\abs{y}}{2} \text{ and } |x-y| < n\right\}.
\end{align*}
For $j \in \left\{0,1,2 \right\}$, we define
\begin{align*}
    K_{n,i,j}&:=\iint_{D_{n,j}} \frac{\left|\varphi_n(x)-\varphi_n(y)\right|^p}{\abs{x-y}^{N+\sigma p}}\left|\frac{\partial u_n}{\partial x_i}(y)\right|^p\dxy.
\end{align*}
Observe that
\begin{align}\label{K-n-i}
    K_{n,i}= K_{n,i,0} + K_{n,i,1} + K_{n,i,2}.
\end{align}
From the arguments as in the proof of \cite[Lemma 3.1]{Density_property}, for each $j \in \left\{0,1,2 \right\}$ we obtain
\begin{align*}
    K_{n,i,j} \leq C \norm{\frac{\partial u_n}{\partial x_i}}_{L_a^{\psig}(B_n^c)}^{\frac{N-\sigma p}{N}},
\end{align*}
where $C$ does not depend on $n$, and hence 
\begin{align}\label{estimate 1}
    K_{n,i,j} &\leq C\left( \norm{\frac{\partial u_n}{\partial x_i}-\frac{\partial u}{\partial x_i}}_{L_a^{\psig}(B_n^c)}^{\frac{N-\sigma p}{N}}+ \norm{\frac{\partial u}{\partial x_i}}_{L_a^{\psig}(B_n^c)}^{\frac{N-\sigma p}{N}}\right)\no \\
    &\leq C\left( \norm{\frac{\partial u_n}{\partial x_i}-\frac{\partial u}{\partial x_i}}_{L_a^{\psig}(\R^N)}^{\frac{N-\sigma p}{N}}+ \norm{\frac{\partial u}{\partial x_i}}_{L_a^{\psig}(B_n^c)}^{\frac{N-\sigma p}{N}}\right).
\end{align}
Since $|\nabla u| \in L^{p_\sigma^*}_a(\R^N)$, using Proposition \ref{bound1}-(ii) we have $ \frac{\partial u_n}{\partial x_i} \in L^{p_\sigma^*}_a(\R^N)$. As $n \ra \infty$, the first term of \eqref{estimate 1} tends to zero using \cite[Lemma 6.1]{Density_property} with $\frac{\partial u}{\partial x_i}$, and the second term of \eqref{estimate 1} tends to zero using  $\frac{\partial u}{\partial x_i} \in \Lasig$. Hence 
\begin{align}\label{k-n-i-convergence}
    \lim\limits_{n \ra \infty} K_{n,i}=0.
\end{align}
Using $0\leq \varphi_n \leq 1$ and the triangle inequality, we have
\begin{align}\label{eq3.12}
    J_{n,i} &\leq C(p) \bigg( \iint_{\R^N \times B_n^c} \frac{\left|\frac{\partial u_n}{\partial x_i}(x)-\frac{\partial u_n}{\partial x_i}(y)-\left(\frac{\partial u}{\partial x_i}(x)-\frac{\partial u}{\partial x_i}(y)\right)\right|^p}{\abs{x-y}^{N+\sigma p}} \dxy \no\\
    &\hspace{3cm} + \iint_{\R^N \times B_n^c} \frac{\left|\frac{\partial u}{\partial x_i}(x)-\frac{\partial u}{\partial x_i}(y)\right|^p}{\abs{x-y}^{N+\sigma p}} \dxy\bigg) \no\\
    &\leq C(p) \bigg( \iint_{\R^N \times \R^N } \frac{\left|\frac{\partial u_n}{\partial x_i}(x)-\frac{\partial u_n}{\partial x_i}(y)-\left(\frac{\partial u}{\partial x_i}(x)-\frac{\partial u}{\partial x_i}(y)\right)\right|^p}{\abs{x-y}^{N+\sigma p}} \dxy \no\\
    &\hspace{3cm} + \iint_{\R^N \times B_n^c} \frac{\left|\frac{\partial u}{\partial x_i}(x)-\frac{\partial u}{\partial x_i}(y)\right|^p}{\abs{x-y}^{N+\sigma p}} \dxy\bigg).
\end{align}
As $n \ra \infty$, the first integral of \eqref{eq3.12} tends to zero using \eqref{e3}, and the second integral of \eqref{eq3.12} tends to zero using 
\begin{align*}
    \iint_{\R^N \times \R^N} \frac{\left|\frac{\partial u}{\partial x_i}(x)-\frac{\partial u}{\partial x_i}(y)\right|^p}{\abs{x-y}^{N+\sigma p}} \dxy < \infty.
\end{align*}
Therefore, 
\begin{align}\label{j-n-i-convergence}
    \lim\limits_{n \ra \infty} J_{n,i}=0.
\end{align}
From \eqref{integral 3.9}, \eqref{k-n-i-convergence} and \eqref{j-n-i-convergence}, we deduce that
\begin{equation}\label{eq5}
    \lim\limits_{n \rightarrow \infty} \iint_{\R^N \times \R^N} \frac{\left|\frac{\partial u_n}{\partial x_i}(x)\varphi_n(x)-\frac{\partial u_n}{\partial x_i}(y)\varphi_n(y)\right|^p}{\abs{x-y}^{N+\sigma p}}\dxy=0.
\end{equation}
We now turn our attention to estimate the final term of \eqref{eq4}. Denote $\psi_{n, i}:= \frac{\partial \varphi_n}{\partial x_i}$. By noticing that $\text{supp}(\psi_{n,i}) \subset \overline{B_{2n}\setminus B_n}$ (since $\text{supp}(\varphi_n) \subset B_n^c$ and $\varphi_n = 1$ in $B_{2n}^c$), for each $i$, we split the following integral
\begin{align}\label{Im}
    I_{n,i} & := \iint_{\R^{N} \times \R^N} \frac{ \left| u_n(x) \psi_{n,i}(x)-u_n(y)\psi_{n,i}(y)  \right|^p}{|x-y|^{N+\sigma p}} \, \dxy \no \\
    &= 2I_{n,i,1} + 2I_{n,i,2} + 2I_{n,i,3} + 2I_{n,i,4} + I_{n,i,5},
\end{align}
where 
\begin{align*}
    I_{n,i,1} &:= \iint_{B_n \times \left(B_{2n} \setminus B_n \right)} \frac{ \left| u_n(x) \psi_{n,i}(x)-u_n(y)\psi_{n,i}(y)  \right|^p}{|x-y|^{N+\sigma p}} \, \dxy\\
    &= \iint_{B_n \times \left(B_{2n} \setminus B_n \right)} \frac{ \left|u_n(y)\psi_{n,i}(y) \right|^p}{|x-y|^{N+\sigma p}} \, \dxy,\\
    I_{n,i,2} &:= \iint_{B_n \times B_{2n}^c } \frac{ \left| u_n(x) \psi_{n,i}(x)-u_n(y)\psi_{n,i}(y)  \right|^p}{|x-y|^{N+\sigma p}} \, \dxy=0 ,\\
    I_{n,i,3} &:= \iint_{\left(B_{2n} \setminus B_n \right) \times \left(B_{2n} \setminus B_n \right)} \frac{ \left| u_n(x) \psi_{n,i}(x)-u_n(y)\psi_{n,i}(y)  \right|^p}{|x-y|^{N+\sigma p}} \, \dxy,\\
    I_{n,i,4} &:= \iint_{\left(B_{2n} \setminus B_n \right) \times  B_{2n}^c } \frac{ \left| u_n(x) \psi_{n,i}(x)-u_n(y)\psi_{n,i}(y)  \right|^p}{|x-y|^{N+\sigma p}} \, \dxy\\
    &=\iint_{\left(B_{2n} \setminus B_n \right) \times B_{2n}^c} \frac{ \left| u_n(x) \psi_{n,i}(x)\right|^p}{|x-y|^{N+\sigma p}} \, \dxy,\\
    I_{n,i,5} &:= \iint_{B_{2n}^c \times B_{2n}^c} \frac{ \left| u_n(x) \psi_{n,i}(x)-u_n(y)\psi_{n,i}(y)  \right|^p}{|x-y|^{N+\sigma p}} \, \dxy =0.
\end{align*}
\textbf{Estimation of $I_{n,i,1}$:} We infer from \eqref{estimate zeta_n} that $\|\nabla^2 \varphi_n\|_{\infty} \leq \frac{C(N)}{n^2}$. Therefore, applying the mean value theorem,
\begin{align}\label{I_ni1 part1}
    I_{n,i,1}  &=  \iint_{B_n \times \left(B_{2n} \setminus B_n \right)} \frac{ \left|u_n(y)\right|^p \left|\psi_{n,i}(x)-\psi_{n,i}(y) \right|^p}{|x-y|^{N+\sigma p}} \, \dxy \no\\
    & \leq \frac{C(N,p)}{n^{2p}} \iint_{B_n \times \left(B_{2n} \setminus B_n \right)} \frac{ \left|u_n(y)\right|^p }{|x-y|^{N+\sigma p-p}} \, \dxy \no\\
     & \leq \frac{C(N,p)}{n^{2p}} \Bigg\{ \iint_{B_{\frac{n}{2}} \times \left(B_{2n} \setminus B_n \right)} \frac{ \left|u_n(y)\right|^p }{|x-y|^{N+\sigma p-p}} \, \dxy \no\\
     &\hspace{3cm} + \iint_{(B_n \setminus B_{\frac{n}{2}}) \times \left(B_{2n} \setminus B_n \right)} \frac{ \left|u_n(y)\right|^p }{|x-y|^{N+\sigma p-p}} \, \dxy\Bigg\}.
\end{align}
Notice that for $x \in B_{\frac{n}{2}}$ and $y \in B_{2n} \setminus B_n$, we have $\abs{y}> \abs{x}$ and
\begin{align*}
     \abs{x-y} \geq \abs{y}-\abs{x} = \frac{\abs{y}}{2}+\frac{\abs{y}}{2}-\abs{x} \geq \frac{\abs{y}}{2}.
\end{align*}
Thus, we obtain
\begin{align}\label{I_ni1 part2}
    &\iint_{B_{\frac{n}{2}} \times \left(B_{2n} \setminus B_n \right)} \frac{ \left|u_n(y)\right|^p }{|x-y|^{N+\sigma p-p}} \, \dxy \leq C \iint_{B_{\frac{n}{2}} \times \left(B_{2n} \setminus B_n \right)} \frac{ \left|u_n(y)\right|^p }{|y|^{N+\sigma p-p}} \, \frac{\dx\dy}{\abs{x}^{2a}} \no\\
    & \leq C \left(\int_{B_{\frac{n}{2}}} \frac{\dx}{\abs{x}^{2a}} \right) \left(\int_{B_{2n} \setminus B_n} \frac{ \left|u_n(y)\right|^p }{|y|^{N+\sigma p-p}}\,\dy  \right) \no \\
    & \leq C n^{N-2a} \left(\int_{B_{2n} \setminus B_n} \frac{ \left|u_n(y)\right|^{\ps} }{|y|^{\frac{2a\ps}{p}}}\,\dy  \right)^{\frac{p}{\ps}} \left(\int_{B_{2n} \setminus B_n} \left(\frac{ \abs{y}^{2a} }{|y|^{N+\sigma p-p}}\right)^{\frac{N}{sp}}\,\dy  \right)^{\frac{sp}{N}}  \text{ (by H\"{o}lder's inequality) } \no \\
    & \leq C n^{N-2a} \frac{ n^{2a}}{n^{N+\sigma p-p}}\left(\int_{B_{2n} \setminus B_n} \frac{ \left|u_n(y)\right|^{\ps} }{|y|^{\frac{2a\ps}{p}}}\,\dy  \right)^{\frac{p}{\ps}} \left(\int_{B_{2n} \setminus B_n} \,\dy  \right)^{\frac{sp}{N}}  \no \\
    & \leq C n^{N-2a} \frac{ n^{2a+sp}}{n^{N+\sigma p-p}}\left(\int_{B_{2n} \setminus B_n} \frac{ \left|u_n(y)\right|^{\ps} }{|y|^{\frac{2a\ps}{p}}}\,\dy  \right)^{\frac{p}{\ps}} = C n^{2p} \left(\int_{B_{2n} \setminus B_n} \frac{ \left|u_n(y)\right|^{\ps} }{|y|^{\frac{2a\ps}{p}}}\,\dy
    \right)^{\frac{p}{\ps}},
\end{align}
where $C=C(N,p,\sigma,a)$.
Further, notice that for $x \in B_n \setminus B_{\frac{n}{2}}$ and $y \in B_{2n} \setminus B_n$, we have $\abs{y}\geq \abs{x}$ and $B_n \setminus B_{\frac{n}{2}} \subset B_{4n}(y)$. 
Hence we estimate 
\begin{align}
   &\iint_{(B_n \setminus B_{\frac{n}{2}}) \times \left(B_{2n} \setminus B_n \right)} \frac{ \left|u_n(y)\right|^p }{|x-y|^{N+\sigma p-p}} \, \dxy \leq \iint_{(B_n \setminus B_{\frac{n}{2}})\times \left(B_{2n} \setminus B_n \right)} \frac{ \left|u_n(y)\right|^p }{|x-y|^{N+\sigma p-p}} \, \frac{\dx\dy}{\abs{x}^{2a}} \no \\
    &\leq \frac{C}{n^{2a}} \int_{B_{2n} \setminus B_n } \abs{u_n(y)}^p  \left(\int_{B_n \setminus B_{\frac{n}{2}}} \frac{\dx}{\abs{x-y}^{N+\sigma p -p}}\right) \,\dy \no\\
    & \leq \frac{C}{n^{2a}} \int_{B_{2n} \setminus B_n } \abs{u_n(y)}^p  \left( \int_{B_{4n}(y)} \frac{\dx}{\abs{x-y}^{N+\sigma p -p}}\right) \,\dy \no\\
    & \leq C \frac{n^{p -\sigma p}}{n^{2a}} \int_{B_{2n} \setminus B_n } \frac{\abs{u_n(y)}^p}{\abs{y}^{2a}} \abs{y}^{2a} \,\dy  \no\\
    & \leq C \frac{n^{p -\sigma p}}{n^{2a}} \left(\int_{B_{2n} \setminus B_n } \frac{\abs{u_n(y)}^{\ps}}{\abs{y}^{\frac{2a \ps}{p}}} \,\dy \right)^{\frac{p}{\ps}} \left(\int_{B_{2n} \setminus B_n } \abs{y}^{\frac{2aN}{sp}} \,\dy \right)^{\frac{sp}{N}} \text{  (by H\"{o}lder's inequality)}  \no \\
    %& \leq \frac{C n^{p -\sigma p +2a}}{n^{2a}}\left(\int_{B_{2n} \setminus B_n } \frac{\abs{u_n(y)}^{\ps}}{\abs{y}^{\frac{2a \ps}{p}}} \,\dy \right)^{\frac{p}{\ps}} \left(\int_{B_{2n} \setminus B_n } \,\dy \right)^{\frac{sp}{N}} \no\\
    & \leq C \frac{n^{p -\sigma p +2a +sp}}{n^{2a}}\left(\int_{B_{2n} \setminus B_n } \frac{\abs{u_n(y)}^{\ps}}{\abs{y}^{\frac{2a \ps}{p}}} \,\dy \right)^{\frac{p}{\ps}} =  C n^{2p} \left(\int_{B_{2n} \setminus B_n } \frac{\abs{u_n(y)}^{\ps}}{\abs{y}^{\frac{2a \ps}{p}}} \,\dy \right)^{\frac{p}{\ps}},
\end{align}
where $C=C(N,p,s,a)$. Hence
\begin{align}\label{I_ni1 part3}
    \iint_{(B_n \setminus B_{\frac{n}{2}}) \times \left(B_{2n} \setminus B_n \right)} \frac{ \left|u_n(y)\right|^p }{|x-y|^{N+\sigma p-p}} \, \dxy \leq C n^{2p}\left(\int_{B_{2n} \setminus B_n } \frac{\abs{u_n(y)}^{\ps}}{\abs{y}^{\frac{2a \ps}{p}}} \,\dy \right)^{\frac{p}{\ps}}.
\end{align}
From \eqref{I_ni1 part1}, \eqref{I_ni1 part2} and \eqref{I_ni1 part3}, we get the following inequality
\begin{align*}
    I_{n,i,1} \leq C(N,p, \sigma,a) \left(\int_{B_{2n} \setminus B_n } \frac{\abs{u_n(y)}^{\ps}}{\abs{y}^{\frac{2a \ps}{p}}} \,\dy \right)^{\frac{p}{\ps}}.
\end{align*}
Using the arguments as in \cite[Lemma 6.1]{Density_property}, we have 
\begin{align}\label{convolution in Weighted Lp}
    \lim\limits_{n \ra \infty}\norm{u - u_n}_{L^{\ps}_a(\R^N)}=0.
\end{align}
Thus,
\begin{align*}
    \int_{B_{2n} \setminus B_n } \frac{\abs{u_n(y)}^{\ps}}{\abs{y}^{\frac{2a \ps}{p}}} \,\dy  &\leq C \int_{B_{2n} \setminus B_n } \frac{\abs{u_n(y)-u(y)}^{\ps}}{\abs{y}^{\frac{2a \ps}{p}}} \,\dy + C\int_{B_{2n} \setminus B_n } \frac{\abs{u(y)}^{\ps}}{\abs{y}^{\frac{2a \ps}{p}}} \,\dy\\
    &\leq C \int_{\R^N } \frac{\abs{u(y)-u_n(y)}^{\ps}}{\abs{y}^{\frac{2a \ps}{p}}} \,\dy + C\int_{B_n^c } \frac{\abs{u(y)}^{\ps}}{\abs{y}^{\frac{2a \ps}{p}}} \,\dy.
\end{align*} 
The first integral on the right-hand side of the above inequality converges to $0$ as $n \ra \infty$ due to \eqref{convolution in Weighted Lp} and the second integral also converges to $0$ as $n \ra \infty$ by dominated convergence theorem. Hence, we conclude that $\lim\limits_{n \ra \infty} I_{n,i,1}=0$.

\noindent \textbf{Estimation of $I_{n,i,4}$:} Using the fact that $\psi_{n,i} =0$ in $B_{2n}^c$, we write
\begin{align*}
    I_{n,i,4} & = \iint_{\left(B_{2n} \setminus B_n \right) \times \left(B_{3n}^c \right)} \frac{ \left| u_n(x) \psi_{n,i}(x) \right|^p}{|x-y|^{N+\sigma p}} \, \dxy \\
    & \hspace{3cm} + \iint_{\left(B_{2n} \setminus B_n \right) \times \left(B_{3n} \setminus B_{2n} \right)} \frac{ \left| u_n(x) \right|^p \left| \psi_{n,i}(x)- \psi_{n,i}(y)\right|^p}{|x-y|^{N+\sigma p}} \, \dxy\\
    & \leq \underbrace{\frac{C}{n^p}\iint_{\left(B_{2n} \setminus B_n \right) \times \left(B_{3n}^c \right)} \frac{ \left| u_n(x) \right|^p }{|x-y|^{N+\sigma p}} \, \dxy}_{A_{n,i}} + \underbrace{\frac{C}{n^{2p}} \iint_{\left(B_{2n} \setminus B_n \right) \times \left(B_{3n} \setminus B_{2n} \right)} \frac{ \left| u_n(x) \right|^p }{|x-y|^{N+\sigma p-p}} \, \dxy}_{B_{n,i}},
\end{align*}
where $C=C(N,p)$. For $A_{n,i}$, we use $\|\nabla \varphi_n\|_{\infty} \leq \frac{C(N)}{n}$, and  $B_{n,i}$ follows using the mean value theorem and $\|\nabla^2 \varphi_n\|_{\infty} \leq \frac{C(N)}{n^2}$. Notice that $B_n(x) \subset B_{3n}$ for $x \in B_{2n} \setminus B_n$. Thus, we get
\begin{align}\label{A-n-i}
    A_{n,i} &= \frac{C(N,p)}{n^p}\int_{B_{2n} \setminus B_n} \left| u_n(x) \right|^p  \left(\int_{B_{3n}^c}\frac{ \dy }{|x-y|^{N+\sigma p} \abs{y}^a} \right)\,\frac{\dx}{\abs{x}^a} \no\\
    & \leq \frac{C(N,p)}{n^{p+a}} \int_{B_{2n} \setminus B_n} \left| u_n(x) \right|^p  \left(\int_{B_n^c(x)}\frac{ \dy }{|x-y|^{N+\sigma p}} \right)\,\frac{\dx}{\abs{x}^a}\no\\
    &\leq \frac{C(N,p, \sigma)}{n^{p+\sigma p+a}}\int_{B_{2n} \setminus B_n} \left| u_n(x) \right|^p \, \frac{\dx}{\abs{x}^a}.
\end{align}
Next, we estimate the second integral, i.e.,
\begin{align*}
   B_{n,i}&= \frac{C(N,p)}{n^{2p}} \iint_{\left(B_{2n} \setminus B_n \right) \times \left(B_{3n} \setminus B_{2n} \right)} \frac{ \left| u_n(x) \right|^p }{|x-y|^{N+\sigma p-p}} \, \dxy\\
   &\le  \frac{C(N,p)}{n^{2p+a}} \int_{B_{2n} \setminus B_n } \left| u_n(x) \right|^p \left(\int_{ B_{3n} \setminus B_{2n}} \frac{ \dy }{|x-y|^{N+\sigma p-p}}  \right)\,\frac{\dx}{\abs{x}^a}.
\end{align*}
Observe that $B_{3n} \subset B_{6n}(x)$ for $x \in B_{2n} \setminus B_n$, and hence 
\begin{align*}
    \int_{ B_{3n} \setminus B_{2n}} \frac{ \dy }{|x-y|^{N+\sigma p-p}}  \leq \int_{ B_{6n}(x)} \frac{ \dy }{|x-y|^{N+\sigma p-p}}  \leq C(N,p, \sigma) n^{p-\sigma p}.
\end{align*}
Thus, we obtain
\begin{align}\label{B-n-i}
   B_{n,i} \le \frac{C(N,p, \sigma)}{n^{p+\sigma p+a}} \int_{B_{2n} \setminus B_n } \left| u_n(x) \right|^p\,\frac{\dx}{\abs{x}^a}.
\end{align}
Further, using H\"{o}lder's inequality, we estimate 
\begin{align}\label{estimate-s}
    &\int_{B_{2n} \setminus B_n } \left| u_n(x) \right|^p\,\frac{\dx}{\abs{x}^a} = \int_{B_{2n} \setminus B_n } \frac{\abs{u_n(x)}^p}{\abs{x}^{2a}} \abs{x}^a \, \dx \no \\
    &\le C n^{a+sp} \left( \int_{B_{2n} \setminus B_n } \frac{\abs{u_n(x)}^{\ps}}{\abs{x}^{\frac{2a\ps}{p}}} \, \dx \right)^{\frac{p}{\ps}} \le C n^{a+sp} \left( \int_{B_n^c } \frac{\abs{u_n(x)}^{\ps}}{\abs{x}^{\frac{2a\ps}{p}}} \, \dx \right)^{\frac{p}{\ps}},
\end{align}
for some $C=C(N,p,s)$. Therefore, in view of \eqref{A-n-i},\eqref{B-n-i}, and using \eqref{convolution in Weighted Lp} and dominated convergence theorem we get 
\begin{align*}
    \lim\limits_{n \rightarrow \infty} A_{n,i} = \lim\limits_{n \rightarrow \infty} B_{n,i} = 0.
\end{align*}
Thus, $\lim\limits_{n \rightarrow \infty} I_{n,i,4}=0$.

\noindent \textbf{Estimation of $I_{n,i,3}$:} We break the integral $I_{n,i,3}$ into two parts as follows
\begin{align*}
    I_{n,i,3} & \leq C(p) \Bigg\{ \underbrace{\iint_{\left(B_{2n} \setminus B_n \right) \times \left(B_{2n} \setminus B_n \right)} \frac{ \left|  \psi_{n,i}(x)-\psi_{n,i}(y)  \right|^p }{|x-y|^{N+\sigma p}}|u_n(x)|^p \, \dxy}_{E_{n,i}} \\
    & \hspace{3cm}+ \underbrace{\iint_{\left(B_{2n} \setminus B_n \right) \times \left(B_{2n} \setminus B_n \right)} \frac{ \left| u_n(x) -u_n(y) \right|^p}{|x-y|^{N+\sigma p}} |\psi_{n,i}(y)|^p\, \dxy}_{F_{n,i}} \Bigg\}.
\end{align*}
Using the mean value theorem and $\|\nabla^2 \varphi_n\|_{\infty} \leq \frac{C(N)}{n^2}$, we estimate $E_{n,i}$ as 
\begin{align*}
   E_{n,i} & \leq \frac{C(N,p)}{n^{2p}} \iint_{\left(B_{2n} \setminus B_n \right) \times \left(B_{2n} \setminus B_n \right)} \frac{ |u_n(x)|^p }{|x-y|^{N+\sigma p-p}} \, \dxy \\
   & \le \frac{C(N,p)}{n^{2p+a}} \int_{B_{2n} \setminus B_n } |u_n(x)|^p \left(\int_{B_{2n} \setminus B_n}\frac{ \dy }{|x-y|^{N+\sigma p-p}} \right)\, \frac{\dx}{\abs{x}^a}.
\end{align*}
Observe that $B_{2n} \subset B_{4n}(x)$, for $x \in B_{2n} \setminus B_n$. Thus using \eqref{estimate-s} we get
\begin{align*}
    E_{n,i} &\leq \frac{C(N,p)}{n^{2p+a}} \int_{B_{2n} \setminus B_n } |u_n(x)|^p \left(\int_{B_{4n}(x)}\frac{ \dy }{|x-y|^{N+\sigma p-p}} \right)\, \frac{\dx}{\abs{x}^a}\\
    %&\leq \frac{C(N,p)}{n^{2p+a}} \int_{B_{2n} \setminus B_n } |u_n(x)|^p \left(\int_{0}^{4n}\frac{ \tau^{N-1} }{\tau^{N+\sigma p-p}} \, \mathrm{d}\tau\right)\,\frac{\dx}{\abs{x}^a}\\
    &\leq \frac{C(N,p, \sigma)}{n^{p+\sigma p+a}} \int_{B_{2n} \setminus B_n } |u_n(x)|^p \,\frac{\dx}{\abs{x}^a} \le C(N,p, \sigma) \left( \int_{B_{2n} \setminus B_n } \frac{\abs{u_n(x)}^{\ps}}{\abs{x}^{\frac{2a\ps}{p}}} \, \dx \right)^{\frac{p}{\ps}}.
\end{align*}
Using \eqref{convolution in Weighted Lp} and subsequent arguments, it follows that $\lim\limits_{n \ra \infty}{E_{n,i}}=0$. Let us now consider the second integral, i.e.,
\begin{align}\label{eq7}
   F_{n,i} &= \iint_{\left(B_{2n} \setminus B_n \right) \times \left(B_{2n} \setminus B_n \right)} \frac{ \left| u_n(x) -u_n(y) \right|^p}{|x-y|^{N+\sigma p}} |\psi_{n,i}(y)|^p\, \dxy \no \\
   & \le \frac{C(N,p)}{n^{p+2a}} \iint_{\left(B_{2n} \setminus B_n \right) \times \left(B_{2n} \setminus B_n \right)} \frac{ \left| u_n(x) -u_n(y) \right|^p}{|x-y|^{N+\sigma p}} \, \dx\dy,
\end{align}
where the inequality holds since $\abs{\psi_{n,i}(y)} \le \frac{C(N)}{n}$. 
Applying Jensen's inequality ($\rho_n \dz$ is a probability measure and $\text{supp}(\rho_n) \subset \overline{B_{\frac{1}{n}}}$), we get 
\begin{align}\label{eq8}
    \left|u_n(x) -u_n(y)\right|^p &= \left|\int_{\R^N} \rho_n(z)u(x-z)\,\mathrm{d}z - \int_{\R^N} \rho_n(z)u(y-z)\,\mathrm{d}z \right|^p \no\\
    &= \left|\int_{B_{\frac{1}{n}}} \left( u(x-z) - u(y-z) \right)\rho_n(z)\,\mathrm{d}z \right|^p \no\\
    &\leq \int_{B_{\frac{1}{n}}} \left|u(x-z) - u(y-z)\right|^p\rho_n(z)\,\mathrm{d}z.
\end{align}
From \eqref{eq7} and \eqref{eq8}, 
\begin{align*}
    F_{n,i} \le \frac{C(N,p)}{n^{p+2a} } \iint_{\left(B_{2n} \setminus B_n \right) \times \left(B_{2n} \setminus B_n \right)} \left( \int_{B_{\frac{1}{n}}} \left|u(x-z) - u(y-z)\right|^p\rho_n(z)\,\mathrm{d}z \right) \frac{\dx\dy}{\abs{x-y}^{N+ \sigma p}}.
\end{align*}
By the fundamental theorem of calculus and Jensen's inequality,
\begin{align}\label{Fub1}
    F_{n,i} & \le \frac{C(N,p)}{n^{p+2a} } \iint_{\left(B_{2n} \setminus B_n \right) \times \left(B_{2n} \setminus B_n \right)} \left( \int_{B_{\frac{1}{n}}} \left( \int_0^{1} \abs{\Gr u(\theta)} \, \dt \right)^{p} \rho_n(z)\,\mathrm{d}z \right) \frac{\dx\dy}{\abs{x-y}^{N+ \sigma p -p}} \no\\
    & \le \frac{C(N,p)}{n^{p+2a} } \iint_{\left(B_{2n} \setminus B_n \right) \times \left(B_{2n} \setminus B_n \right)} \left( \int_{B_{\frac{1}{n}}} \left( \int_0^{1} \abs{\Gr u(\theta)}^p \, \dt \right) \rho_n(z)\,\mathrm{d}z \right) \frac{\dx\dy}{\abs{x-y}^{N+ \sigma p -p}},
\end{align}
where $\theta = t(x-z)+(1-t)(y-z)$.
Consider the following function
\begin{align*}
    A(t,x,y,z) = \frac{ \left| \Gr u(\theta) \right|^p}{|x-y|^{N+\sigma p-p}}  \rho_n(z), \text{ where } x,y \in B_{2n} \setminus B_n, z \in B_{\frac{1}{n}} \text{ and } t\in (0,1). 
\end{align*}
Observe that $\theta = \theta(t,x,y,z)$ and
\begin{align}\label{bound for theta}
\abs{\theta} \le t \left( \abs{x} + \abs{z} \right) + (1-t) \left(  \abs{y} + \abs{z} \right) \le 2n + \frac{1}{n}.
\end{align}
Now for fixed $t \in (0,1)$ and $z \in B_{\frac{1}{n}}$ we estimate the following integral 
\begin{align*}
   \iint_{\left(B_{2n} \setminus B_n \right) \times \left(B_{2n} \setminus B_n \right)} \frac{\abs{\Gr u(\theta(t, \cdot, \cdot, z))}^p}{|x-y|^{N+\sigma p-p}} \, \dx \dy.
\end{align*}
For fixed $x \in B_{2n} \setminus B_n$, we consider the following function 
\begin{align*}
  f_1(y) = \frac{\abs{\Gr u(\theta(t,x, \cdot, z))}^p}{\abs{\theta(t,x, \cdot, z)}^{2a}},  \text{ for } y \in \R^N.
\end{align*}
Then using $\norm{\Gr u}_{L_a^{p^*_{\sigma}}(\R^N)} < \infty$ we see that 
\begin{align*}
    \int_{\R^N} f_1(y)^{\frac{p^*_{\sigma}}{p}} \, \dy \le (1-t)^{-N} \int_{\R^N} \frac{\abs{\Gr u(\theta(t,x, \cdot, z))}^{p^*_{\sigma}}}{\abs{\theta(t,x, \cdot, z)}^{\frac{2ap^*_{\sigma}}{p}}} \, \rm{d} \theta < \infty.  
\end{align*}
Also by Example \ref{example-weak-Lp}, we have 
\begin{align*}
    g_1(\eta) = \frac{1}{\abs{\eta}^{N+\sigma p-p}} \in L^{{\frac{N}{N+\sigma p-p}},\infty}(\R^N).
\end{align*}
Therefore, by Proposition \ref{convolution weak}, $h_1(\eta) = \int_{\R^N} g_1(\eta-y) f_1(y)\, \dy < \infty $ for a.e. $\eta \in \R^N$ and moreover
\begin{align*}
    h_1 \in L^r(\R^N), \text{ where } \frac{1}{r} = \frac{p}{p^*_{\sigma}} + \frac{N+\sigma p-p}{N} -1 = \frac{N-p}{N}.
\end{align*}
Hence for $x \in B_{2n} \setminus B_n$ using \eqref{bound for theta} we get 
\begin{align}\label{h_2 bound}
  h_2(x) := \int_{B_{2n} \setminus B_n} \abs{\Gr u(\theta(t,x, \cdot, z))}^p  g_1(x-y) \, \dy & \le c_1(n) \int_{B_{2n} \setminus B_n} g_1(x-y) f_1(y)\, \dy \no \\
  & \le c_1(n) h_1(x),
\end{align}
where $c_1(n) = \left( 2n + \frac{1}{n} \right)^{2a}$. Further, using the H\"{o}lder's inequality, 
\begin{align*}
    \int_{B_{2n}\setminus B_n} h_2(x) \, \dx \le C(N,p) c_1(n) \norm{h_1}_{r} n^{\frac{N}{r'}}.
\end{align*}
Therefore, applying Tonelli's theorem, for a.e. $z \in B_{\frac{1}{n}}$ and $t \in (0,1)$ we get 
\begin{align}\label{Tonelli_2}
   B(t,z):= \iint_{\left(B_{2n} \setminus B_n \right) \times \left(B_{2n} \setminus B_n \right)} \frac{ \left| \nabla u(\theta(t,\cdot, \cdot, z))\right|^p}{|x-y|^{N+\sigma p-p}} \, \dx \dy < \infty.
\end{align}
Further, Fubini's theorem yields
\begin{align}\label{eq7.5}
    B(t,z) &= \int_{B_{2n} \setminus B_n} \left( \int_{B_{2n} \setminus B_n} \frac{ \left| \nabla u(\theta(t,\cdot, \cdot, z))\right|^p}{|x-y|^{N+\sigma p-p}} \, \dy \right) \, \dx \no \\
    & =  \int_{B_{2n} \setminus B_n} h_2(x)\, \dx \leq c_1(n) \int_{B_{2n} \setminus B_n} h_1(x)\, \dx \text{  (using \eqref{h_2 bound})} \no \\
    & \leq c_1(n) \left(\int_{B_{2n} \setminus B_n} h_1(x)^r\, \dx \right)^{\frac{1}{r}} \left(\int_{B_{2n} \setminus B_n} \, \dx \right)^{\frac{p}{N}} \no \\ 
    & \leq C(N,p) c_1(n) n^p \left(\int_{B_n^c} h_1(x)^r\, \dx \right)^{\frac{1}{r}}.
\end{align}
Notice that $h_1(x) \equiv h_1(t,x,z)$, and for every $z \in B_{\frac{1}{n}}$ and $t\in (0,1)$, since $h_1(t,\cdot, z) \in L^r(\R^N)$, we have  
\begin{align}\label{convergence of h_1}
    \left(\int_{B_n^c} h_1(t,\cdot,z)^r\, \dx \right)^{\frac{1}{r}} \ra 0, \text{ as } n \ra \infty.
\end{align} 
Then there exists $n_0 \in \N$ such that for $n \ge n_0$, $\int_{B_n^c} h_1(t,\cdot,z)^r\, \dx  < 1$ for every $t\in (0,1)$ and $z \in B_{\frac{1}{n}}$. Therefore,
\begin{align*}
    f(t) := \int_{B_{\frac{1}{n}}} \rho_n(z)  B(t,z) \, \dz & \le C(N,p) c_1(n) n^p \int_{B_{\frac{1}{n}}} \rho_n(z) \left(\int_{B_n^c} h_1(t,\cdot, z)^r\, \dx \right)^{\frac{1}{r}} \, \dz \\
    & \le C(N,p) c_1(n) n^p \int_{B_{\frac{1}{n}}} \rho_n(z) \, \dz  =  C(N,p) c_1(n) n^p,
\end{align*}
and $ \int_0^1 f(t) \, \dt \le C(N,p) c_1(n) n^p$.
Hence for $n \ge n_0$ using Tonneli's theorem,
$$A \in L^1\left( \left(B_{2n} \setminus B_n  \right) \times\left(B_{2n} \setminus B_n  \right) \times B_{\frac{1}{n}} \times (0,1) \right),$$
and subsequently using Fubini's theorem, from \eqref{Fub1} we obtain
\begin{align}\label{eq9}
   F_{n,i} \leq \frac{C(N,p)}{n^{p+2a}}\int_0^1 \left( \int_{B_{\frac{1}{n}}} \left( \iint_{\left(B_{2n} \setminus B_n \right) \times \left(B_{2n} \setminus B_n \right)} \frac{ \left| \Gr u(\theta) \right|^p}{|x-y|^{N+\sigma p-p}} \, \dx \dy \right)\rho_n(z)\,\mathrm{d}z\right) \,{\rm d}t.
\end{align}
Take $\ep>0$. Then using the convergence in \eqref{convergence of h_1} there exists $n_1 (\ge n_0) \in \N$ such that for $n \ge n_1$, using \eqref{eq7.5} and \eqref{eq9}, we conclude
\begin{align*}
     F_{n,i} &\leq C(N,p) \frac{c_1(n) n^p}{n^{p+2a}} \int_0^1 \left(\int_{B_{\frac{1}{n}}} \left( \int_{B_n^c} h_1(t,\cdot,z)^r\, \dx  \right)^{\frac{1}{r}} \rho_n(z)\,\mathrm{d}z\right)\,{\rm d}t \\
     &\le C(N,p) \int_0^1 \left(\int_{B_{\frac{1}{n}}} \left( \int_{B_n^c} h_1(t,\cdot,z)^r\, \dx  \right)^{\frac{1}{r}} \rho_n(z)\,\mathrm{d}z \right)\,{\rm d}t< C(N,p) \ep. 
\end{align*}
Therefore, $\lim\limits_{n \ra \infty} F_{n,i}=0$. In conclusion, we get $\lim\limits_{n \rightarrow \infty} I_{n,i,3}=0$. 

Finally, taking limit as $n \rightarrow \infty$ in \eqref{Im} and using the convergence of $I_{n,i,k}, k=1,\cdot, \cdot, \cdot,5$, we get
 \begin{align}\label{lim_Im}
     \lim\limits_{n \rightarrow \infty}I_{n,i}=0, \text{ for each } i=1,2,\cdot, \cdot, \cdot, N.
 \end{align} 
 Therefore, by \eqref{eq4}, \eqref{eq5} and \eqref{lim_Im}, we obtain
 \begin{align}\label{eq1.2}
    \lim\limits_{n \rightarrow \infty} [ u \star \rho_n-\left(u \star \rho_n\right) \zeta_n]_{s,p,a} =0.
 \end{align}
 From \eqref{eq1}, \eqref{eq1.1} and \eqref{eq1.2}, we get the required result.
\end{proof}

\begin{remark}\label{subset}
    In view of Proposition \ref{Prop3.2} and the definition of $\Dsp$, we infer that 
    $\Wsp \subset \Dsp$. More precisely, if $u \in \Wsp$, then using Proposition \ref{Prop3.2} there exists $\{u_n\} \subset \cc(\R^N)$ such that $[u-u_n]_{s,p,a} \ra 0$ as $n \ra \infty$, which implies that $u \in \Dsp$.
\end{remark}
In the following proposition, we prove that the seminorm $\left[\cdot\right]_{s,p,a}$ is an equivalent norm in $\Wsp$. 
\begin{proposition}\label{prop3.4}
   Let $N \ge 2, s \in (1,2), \sigma = s-1, p \in (1, \frac{N}{s})$, and $a \in [0, \frac{N-sp}{2})$. Then there exists $C=C(N,p, \sigma,a)>0$ such that 
    \begin{align*}
        \norm{\Gr u}_{\Lasig} \leq C \left[u\right]_{s,p,a}, \; \forall \, u \in \Wsp.  
    \end{align*}   
\end{proposition}
\begin{proof}
For $u \in \Wsp$, we consider the same sequence $v_n = \left(u \star \rho_n \right)\zeta_n$ for $n \in \N$ as in Proposition \ref{Prop3.2}. Then $v_n \in \cc(\R^N)$. Using \cite[Theorem 1.5]{Caffarelli_Kohn_2017} with $\frac{\del v_n}{\del x_i}$ and Lemma \ref{lemma_1}, we get 
\begin{align}\label{eq2.1}
    \norm{\Gr v_n}_{\Lasig} \leq C(N,p, \sigma,a) \left[v_n\right]_{s,p,a}, \; \forall \, n \in \N.
\end{align}
We claim 
\begin{align}\label{grad_un psig norm}
    \norm{\Gr u -\Gr v_n}_{\Lasig} \ra 0, \text{ as } n \ra \infty.
\end{align} 
Using the triangle inequality, we obtain
\begin{align}\label{grad_un psig norm esti}
    &\norm{\Gr u -\Gr v_n}_{\Lasig} \no\\
    &= \norm{\Gr u-\left(\Gr u \star \rho_n\right)\zeta_n - \left( u \star \rho_n\right)\Gr \zeta_n }_{\Lasig} \no\\
    %&\leq \norm{\left(\Gr u \star \rho_n\right)\zeta_n -\Gr u}_{\Lasig} + \norm{ \left( u \star \rho_n\right)\Gr \zeta_n }_{\Lasig} \no\\
    & \leq \underbrace{\norm{\Gr u \star \rho_n-\left(\Gr u \star \rho_n\right)\zeta_n }_{\Lasig}}_{A_n}+ \underbrace{ \norm{\Gr u-\Gr u \star \rho_n }_{\Lasig}}_{B_n} + \underbrace{\norm{ \left( u \star \rho_n\right)\Gr \zeta_n }_{\Lasig}}_{C_n}.
\end{align}
As before, we define $\varphi_n = 1 -\zeta_n$. By noticing $0 \leq \varphi_n \leq 1$, we deduce 
\begin{align}\label{eq2.2}
    A_n^{\psig}&= \norm{\left(\Gr u \star \rho_n\right)\varphi_n}_{\Lasig}^{\psig}= \int_{\R^N} \frac{\left|\left(\Gr u \star \rho_n\right)\varphi_n\right|^{\psig}}{|x|^{\frac{2a \psig}{p}}}\dx = \int_{\R^N} \left(\sum\limits_{i=1}^{N}\left|\frac{\partial u_n}{\partial x_i}\varphi_n\right|^2 \right)^{\frac{\psig}{2}} \frac{\dx}{|x|^{\frac{2a \psig}{p}}} \no\\
    & \leq A(N,\frac{\psig}{2}) \sum\limits_{i=1}^{N}\int_{\R^N} \left|\frac{\partial u_n}{\partial x_i}\varphi_n\right|^{\psig}\frac{\dx}{|x|^{\frac{2a \psig}{p}}} \text{ (using Lemma \ref{lemma_1}-(i))}\no\\
    & \leq 2^{\psig -1} A(N,\frac{\psig}{2}) \left(\sum\limits_{i=1}^{N}\int_{\R^N} \left|\left(\frac{\partial u_n}{\partial x_i}- \frac{\partial u}{\partial x_i} \right)\varphi_n\right|^{\psig}\,\frac{\dx}{|x|^{\frac{2a \psig}{p}}} + \sum\limits_{i=1}^{N}\int_{\R^N} \left| \frac{\partial u}{\partial x_i} \varphi_n\right|^{\psig}\,\frac{\dx}{|x|^{\frac{2a \psig}{p}}} \right)\no\\
    & \leq 2^{\psig -1} A(N,\frac{\psig}{2}) \left(\sum\limits_{i=1}^{N}\int_{\R^N} \left|\frac{\partial u}{\partial x_i} - \frac{\partial u_n}{\partial x_i} \right|^{\psig}\frac{\dx}{|x|^{\frac{2a \psig}{p}}} + \sum\limits_{i=1}^{N}\int_{\R^N} \left| \frac{\partial u}{\partial x_i} \varphi_n\right|^{\psig}\,\frac{\dx}{|x|^{\frac{2a \psig}{p}}} \right).
\end{align}
Since $|\Gr u| \in \Lasig$, we get $\frac{\partial u}{\partial x_i} \in \Lasig$. Thus, applying \cite[Lemma 6.1]{Density_property} with $\frac{\del u}{\del x_i}$, we have
\begin{align}\label{eq2.3}
    \lim\limits_{n \ra \infty} \int_{\R^N} \left|\frac{\partial u}{\partial x_i} - \frac{\partial u_n}{\partial x_i} \right|^{\psig}\,\frac{\dx}{|x|^{\frac{2a \psig}{p}}} = 0, \text{ for each } i=1,2,\cdot,\cdot,\cdot,N.
\end{align}
By noticing that $\text{supp}(\zeta_n) \subset \overline{B_{2n}}$ and $\text{supp}(\varphi_n) \subset B_n^c$, the dominated convergence theorem yields
\begin{align}\label{eq2.4}
    \lim\limits_{n \ra \infty} \int_{\R^N} \left| \frac{\partial u}{\partial x_i} \varphi_n\right|^{\psig}\,\frac{\dx}{|x|^{\frac{2a \psig}{p}}} = 0, \text{ for each } i=1,2,\cdot,\cdot,\cdot,N.
\end{align}
Considering \eqref{eq2.2}, \eqref{eq2.3}, and \eqref{eq2.4}, we can deduce that
\begin{align}\label{lim_An}
    \lim\limits_{n \ra \infty} A_n =0.
\end{align}
Further applying Lemma \ref{lemma_1}-(i), we derive
\begin{align*}
    B_n^{\psig} = \norm{\Gr u- \Gr u \star \rho_n}_{\Lasig}^{\psig} & = \int_{\R^N} \left(\sum\limits_{i=1}^{N}\left|\frac{\partial u}{\partial x_i}-\frac{\partial u_n}{\partial x_i}\right|^2 \right)^{\frac{\psig}{2}}\,\frac{\dx}{|x|^{\frac{2a \psig}{p}}} \\
    & \leq A(N, \frac{\psig}{2}) \sum\limits_{i=1}^{N}\int_{\R^N} \left|\frac{\partial u}{\partial x_i}-\frac{\partial u_n}{\partial x_i}\right|^{\psig} \,\frac{\dx}{|x|^{\frac{2a \psig}{p}}}.
\end{align*}
We use \eqref{eq2.3} in the above inequality to get
\begin{align}\label{lim_Bn}
    \lim\limits_{n \ra \infty} B_n = 0.
\end{align}
Using $\norm{\Gr \zeta_n}_{\infty} \leq \frac{C}{n}$ and $\text{supp}(\abs{\Gr \zeta_n}) \subset \overline{B_{2n} \setminus B_n}$, we estimate the last term of \eqref{grad_un psig norm esti} as follows
\begin{align}\label{c_n}
    C_n^{\psig}&= \int_{\R^N} \left|\left( u \star \rho_n\right)\Gr \zeta_n \right|^{\psig}\,\frac{\dx}{|x|^{\frac{2a \psig}{p}}} \no \\
    & \leq \frac{C(N,p, \sigma)}{n^{\psig}}\int_{B_{2n} \setminus B_n} \left| u \star \rho_n \right|^{\psig}\,\frac{\dx}{|x|^{\frac{2a \psig}{p}}} \leq \frac{C(N,p, \sigma)}{n^{\psig +\frac{2a \psig}{p}}}\int_{B_{2n} \setminus B_n} \left| \int_{B_{\frac{1}{n}}}u(x-z) \rho_n(z)\,\dz \right|^{\psig}\,\dx \no \\
    & \leq \frac{C(N,p, \sigma)}{n^{\psig +\frac{2a \psig}{p}}}\int_{B_{2n} \setminus B_n} \left( \int_{B_{\frac{1}{n}}}\left|u(x-z)\right|^{\psig} \rho_n(z)\,\dz \right)\,\dx.
\end{align}
The last inequality follows using Jensen's inequality ($\rho_n \dz$ is a probability measure and $\text{supp}(\rho_n) \subset \overline{B_{\frac{1}{n}}}$).
For $x \in B_{2n} \setminus B_n$ and $z \in B_{\frac{1}{n}}$, we have
\begin{align}\label{xz inequality}
    n-\frac{1}{n}\leq\abs{x}-\abs{z}\leq \abs{x-z} \leq \abs{x} +\abs{z} \leq 2n + \frac{1}{n}.
\end{align}
Now consider $F(x,z) = \left|u(x-z)\right|^{\psig} \rho_n(z)$. 
Then for a.e. $z \in B_{\frac{1}{n}}$, we have
\begin{align*}
    \int_{B_{2n} \setminus B_n} F(x,z)\,\dx &= \rho_n(z)\int_{B_{2n} \setminus B_n}\left|u(x-z)\right|^{\psig}\,\dx \\
    & \leq c_2(n) \rho_n(z)\int_{B_{2n} \setminus B_n}\frac{\left|u(x-z)\right|^{\psig}}{\abs{x-z}^{\frac{2a \psig}{p}}}\,\dx, \text{ where $c_2(n)= \left(2n + \frac{1}{n} \right)^{\frac{2a \psig}{p}}$}\\
    & \leq c_2(n) \rho_n(z) \left(\int_{B_{2n} \setminus B_n}\frac{\left|u(x-z)\right|^{\ps}}{\abs{x-z}^{\frac{2a \ps}{p}}}\,\dx \right)^{\frac{\psig}{\ps}} \left(\int_{B_{2n} \setminus B_n}\,\dx \right)^{\frac{p}{N-\sigma p}}.
\end{align*} 
Thus we obtain
\begin{align*}
     \int_{B_{2n} \setminus B_n} F(x,z) \,\dx \leq C(N,p,\sigma)c_2(n) n^{\psig} \rho_n(z) \left(\int_{B_n^c}\frac{\left|u(x-z)\right|^{\ps}}{\abs{x-z}^{\frac{2a \ps}{p}}}\,\dx \right)^{\frac{\psig}{\ps}}.
\end{align*} 
Since $u \in \Lpastar$, for arbitrary $0<\ep<1$, there exists $n_\ep \in \N$ such that
\begin{align}\label{cgce_1}
    \left(\int_{B_n^c}\frac{\left|u(x-z)\right|^{\ps}}{\abs{x-z}^{\frac{2a \ps}{p}}}\,\dx \right)^{\frac{\psig}{\ps}} < \ep, \; \forall \, n \geq n_\ep.
\end{align}
Using \eqref{cgce_1}, for $n \geq n_\ep$ we obtain
\begin{align*}
    \int_{B_{2n} \setminus B_n} F(x,z) \,\dx \leq  C(N,p,\sigma)c_3(n) \rho_n(z) \ep,
\end{align*}
where $c_3(n) = n^{\frac{2a \psig}{p}+\psig}$. Further, we see that for every $n \geq n_\ep$,
\begin{align}\label{3.39}
   \int_{B_{\frac{1}{n}}} \left(\int_{B_{2n} \setminus B_n} F(x,z) \,\dx\right)\,\dz \leq  C(N,p,\sigma)c_3(n) \ep \int_{B_{\frac{1}{n}}}\rho_n(z)\,\dz = C(N,p,\sigma)c_3(n) \ep.
\end{align}
Therefore, Tonelli's theorem yields $F \in L^1\left(\left({B_{2n} \setminus B_n}\right) \times B_{\frac{1}{n}}\right)$. Subsequently, applying Fubini's theorem in \eqref{c_n} and then using \eqref{3.39} we get
\begin{align*}
    C_n^{\psig} \leq \frac{C(N,p, \sigma)}{n^{\psig +\frac{2a \psig}{p}}} \int_{B_{\frac{1}{n}}} \left(\int_{B_{2n} \setminus B_n}  \left|u(x-z)\right|^{\psig} \rho_n(z)\,\dx \right)\,\dz < C(N,p,\sigma) \ep , \; \forall \, n \geq n_\ep.
\end{align*}
Therefore,  
\begin{align}\label{lim_Cn}
    \lim\limits_{n \ra \infty} C_n =0.
\end{align}
From \eqref{grad_un psig norm esti}, \eqref{lim_An}, \eqref{lim_Bn} and \eqref{lim_Cn}, we obtain  \eqref{grad_un psig norm}. Moreover, we know from Proposition \ref{Prop3.2} that
\begin{align}\label{usp_cgce}
    \left[ u-v_n\right]_{s,p,a} \ra 0, \text{ as } n \ra \infty. 
\end{align}
Next, passing the limit to the inequality \eqref{eq2.1} and using \eqref{grad_un psig norm} and \eqref{usp_cgce}, we get
\begin{align*}
    \norm{\Gr u}_{\Lasig} \leq C \left[u\right]_{s,p,a}, \; \forall \, u \in \Wsp,
\end{align*}
which completes the proof.
\end{proof}

\noi \textbf{Proof of Theorem \ref{Theorem1.1}-((i) and (ii)):}
The proof of part-(i) and part-(ii) is followed from Proposition \ref{Prop3.2} and Proposition \ref{prop3.4}, respectively. \qed

\section{Characterization of fractional weighted homogeneous spaces}

Now, we proceed to characterize the function space $\Dsp$. 
In the following proposition, we show that every Cauchy sequence in $\cc(\R^N)$ converges to a function in $\Wsp$ with respect to the Gagliardo seminorm $[\cdot]_{s,p,a}.$

\begin{proposition}\label{Prop4.1}
Let $N \ge 2, s \in (1,2), \sigma = s-1, p \in (1, \frac{N}{s})$, and $a \in [0, \frac{N-sp}{2})$. Let $\{w_n\} \subset \cc(\R^N)$ be a Cauchy sequence with respect to $[\cdot]_{s,p,a}$. Then there exists $w \in \Wsp$ such that
\begin{align*}
    [w-w_n]_{s,p,a} \ra 0, \text{ as } n \ra \infty,
\end{align*}
\end{proposition}
\begin{proof}
Let $\ep > 0$ be given. By hypothesis, there exists  $n_\epsilon \in \N$ such that
\begin{align}\label{Cauchy1}
    \left[w_n - w_m \right]_{s,p,a} < \epsilon, \; \forall \, m,n \geq n_\epsilon.
\end{align}
By \cite[Theorem 1]{Caffarelli_kohn_1984}, we have
\begin{align}\label{Caffarelli_1984}
    \norm{u}_{\Lpastar} \leq C_1(N,p,\sigma,a) \norm{\Gr u}_{\Lasig}, \; \forall \, u \in \cc(\R^N).
\end{align}
Further, by Proposition \ref{prop3.4}, 
\begin{align}\label{inequality4.3}
    \norm{\Gr u}_{\Lasig} \leq C_2(N,p,\sigma,a) [u]_{s,p,a}, \; \forall \, u \in \cc(\R^N),
\end{align}
From \eqref{Cauchy1}, \eqref{Caffarelli_1984}, \eqref{inequality4.3}, and the fact that $\{w_n\} \subset \cc(\R^N)$, we get $\{w_n\}$ is Cauchy in 
$\Lpastar$ and  $\left\{\frac{\del w_n}{\del x_i}\right\}$ is Cauchy in $\Lasig$. Consequently, $w_n \ra w$ in $\Lpastar$ and $\frac{\del w_n}{\del x_i} \ra v_i$ in $\Lasig$. Moreover, up to a subsequence, $w_n \ra w$ and $\frac{\del w_n}{\del x_i} \ra v_i$  a.e. in $\R^N$. To show $v_i = \frac{\del w}{\del x_i}$, we consider $\varphi \in \cc(\R^N)$ with $\text{supp}(\varphi) \subset K$, where $K$ is a bounded open set of class $\mathcal{C}^{1,1}$. The integration by parts formula yields
\begin{align}\label{integration by parts}
    \int_{\R^N} w_n \frac{\del \varphi}{\del x_i}\,\dx =\int_{K} w_n \frac{\del \varphi}{\del x_i}\,\dx = - \int_{K}  \frac{\del w_n}{\del x_i} \varphi\,\dx=- \int_{\R^N}  \frac{\del w_n}{\del x_i} \varphi\,\dx.
\end{align}
Using H\"{o}lder's inequality and $w_n \ra w$ in $\Lpastar$, we get
\begin{align}\label{lim1}
    &\left|  \int_{\R^N} \left( w -w_n \right) \frac{\del \varphi}{\del x_i}\,\dx\right| \leq \norm{\frac{\del \varphi}{\del x_i}}_\infty \int_{K} \left|   w -w_n \right| \,\dx \no\\
    & = \norm{\frac{\del \varphi}{\del x_i}}_\infty \int_{K} \frac{\left|   w -w_n \right|}{|x|^{\frac{2a}{p}}}  \cdot |x|^{\frac{2a}{p}} \,\dx \leq \norm{\frac{\del \varphi}{\del x_i}}_\infty \left(\int_{K} \frac{\left|   w -w_n \right|^{\ps}}{|x|^{\frac{2a \ps}{p}}}  \,\dx \right)^{\frac{1}{\ps}} \left(\int_{K}  \left(|x|^{\frac{2a}{p}}\right)^\frac{\ps}{\ps -1} \,\dx \right)^{\frac{\ps -1}{\ps}} \no\\
   & \leq C(N,p,\sigma,a,K)\norm{\frac{\del \varphi}{\del x_i}}_\infty \norm{w-w_n}_{\Lpastar} \ra 0, \text{ as } n \ra \infty.
\end{align}
Using $\frac{\del w_n}{\del x_i} \ra v_i$ in $\Lasig$ and similar calculations as in \eqref{lim1}, we also get
\begin{align}\label{lim2}
    &\left|  \int_{\R^N} \left( v_i-\frac{\del w_n}{\del x_i} \right)  \varphi\,\dx  \right| \ra 0, \text{ as } n \ra \infty.
\end{align}
Taking the limit as $n \ra \infty$ in \eqref{integration by parts}, and using \eqref{lim1} and \eqref{lim2}, 
\begin{align*}
     \int_{\R^N} w \frac{\del \varphi}{\del x_i}\,\dx = - \int_{\R^N}  v_i \varphi\,\dx, \; \forall \, \varphi \in \cc(\R^N),
\end{align*}
which implies $\frac{\del w}{\del x_i} = v_i \in \Lasig$, and hence $\abs{\Gr w} \in \Lasig$ (by Lemma \ref{lemma_1}-(i)).
For $n \geq n_\epsilon$, the triangle inequality and \eqref{Cauchy1} give 
\begin{align}\label{bddness of w_n}
    \left[w_n\right]_{s,p,a}\leq \left[w_n - w_{n_\epsilon}\right]_{s,p,a} + \left[ w_{n_\epsilon}\right]_{s,p,a}< \epsilon + \left[ w_{n_\epsilon}\right]_{s,p,a}.
\end{align}
For a.e. $(x,y) \in \R^N \times \R^N$, we define 
\begin{align}\label{V_n and V}
    V_n(x,y)= \frac{\left|\Gr w_n(x)- \Gr w_n(y)\right|^p}{\left| x-y\right|^{N+\sigma p} \abs{x}^a \abs{y}^a} \text{ and } V(x,y)= \frac{\left|\Gr w(x)- \Gr w(y)\right|^p}{\left| x-y\right|^{N+\sigma p} \abs{x}^a \abs{y}^a}.
\end{align}
Using \eqref{bddness of w_n}, $\{V_n\}$ is bounded in $L^1(\R^N \times \R^N)$. Since  $\frac{\del w_n}{\del x_i} \ra \frac{\del w}{\del x_i} $ a.e. in $\R^N$, we get $V_n \ra V$ a.e. in $\R^N \times \R^N$. Hence the Fatou's Lemma yields
\begin{align*}
  \left[w\right]_{s,p,a}^p = \iint_{\R^N \times \R^N} V(x,y)\,\dx \dy \leq  \liminf_{n \ra \infty} \iint_{\R^N \times \R^N} V_n(x,y)\,\dx \dy < \infty.
\end{align*}
Therefore, $w \in \Wsp$. Now we show that $w_n \ra w$ in $\Dsp$.
We define 
\begin{align*}
    V_{n,i}(x,y) = \frac{\frac{\del w_n}{\del x_i}(x)-\frac{\del w_n}{\del x_i}(y)}{\left|x-y \right|^{\frac{N}{p}+\sigma } \abs{x}^{\frac{a}{p}}\abs{y}^{\frac{a}{p}}},  \text{ for a.e. } (x, y) \in \R^N \times \R^N.
\end{align*}
Using Lemma \ref{lemma_1}-(ii), we have
\begin{align}\label{Cauchy2}
    &\left[w_n - w_m\right]_{s,p,a}^p = \iint_{\R^N \times \R^N} \frac{\left|\left(\Gr w_n(x)-\Gr w_n(y)\right) -\left(\Gr w_m(x)-\Gr w_m(y)\right) \right|^p}{\left|x-y \right|^{N+\sigma p}}\,\dxy \no \\
    %&=\iint_{\R^N \times \R^N} \frac{\left[\sum_{i=1}^{N}\left|\left(\frac{\del w_n}{\del x_i}(x)-\frac{\del w_n}{\del x_i}(y)\right) -\left(\frac{\del w_m}{\del x_i}(x)-\frac{\del w_m}{\del x_i}(y)\right) \right|^2\right]^{\frac{p}{2}}}{\left|x-y \right|^{N+\sigma p}}\,\dxy \no \\
    & \geq B(N, \frac{p}{2})\left(\sum_{i=1}^{N}\iint_{\R^N \times \R^N} \frac{\left|\left(\frac{\del w_n}{\del x_i}(x)-\frac{\del w_n}{\del x_i}(y)\right) -\left(\frac{\del w_m}{\del x_i}(x)-\frac{\del w_m}{\del x_i}(y)\right) \right|^p}{\left|x-y \right|^{N+\sigma p}}\,\dxy \right) \no \\
    & = B(N, \frac{p}{2}) \sum_{i=1}^{N} \norm{V_{n,i} - V_{m,i}}_{L^p(\R^N \times \R^N)}^p.
\end{align}
From \eqref{Cauchy1} and \eqref{Cauchy2}, we deduce that for each $i$, the sequence $\{V_{n,i}\}$ is Cauchy in $L^p(\R^N \times \R^N)$. By completeness of $L^p(\R^N \times \R^N)$, there exists $V_i \in L^p(\R^N \times \R^N)$ such that $V_{n,i} \ra V_i \text{ in } L^p(\R^N \times \R^N)$. 
Also, $\frac{\del w_n}{\del x_i} \ra \frac{\del w}{\del x_i}$ a.e. in $\R^N$. Therefore, to the uniqueness of the limit, 
\begin{align*}
    V_i(x,y) = \frac{\frac{\del w}{\del x_i}(x)-\frac{\del w}{\del x_i}(y)}{\left|x-y \right|^{\frac{N}{p}+\sigma } \abs{x}^{\frac{a}{p}}\abs{y}^{\frac{a}{p}}}, \text{ for a.e. } (x,y) \in \R^N \times \R^N.
\end{align*}
Therefore,
\begin{align*}
    \left[w - w_n\right]_{s,p,a}^p &= \iint_{\R^N \times \R^N} \frac{\left|\left(\Gr w(x)-\Gr w(y)\right) -\left(\Gr w_n(x)-\Gr w_n(y)\right) \right|^p}{\left|x-y \right|^{N+\sigma p}}\,\dxy\\
    %&=\iint_{\R^N \times \R^N} \frac{\left[\sum_{i=1}^{N}\left|\left(\frac{\del w_n}{\del x_i}(x)-\frac{\del w_n}{\del x_i}(y)\right) -\left(\frac{\del w}{\del x_i}(x)-\frac{\del w}{\del x_i}(y)\right) \right|^2\right]^{\frac{p}{2}}}{\left|x-y \right|^{N+\sigma p}}\,\dxy\\
     &\leq A(N,\frac{p}{2}) \left( \sum_{i=1}^{N}\iint_{\R^N \times \R^N} \frac{\left|\left(\frac{\del w}{\del x_i}(x)-\frac{\del w}{\del x_i}(y)\right) -\left(\frac{\del w_n}{\del x_i}(x)-\frac{\del w_n}{\del x_i}(y)\right) \right|^p}{\left|x-y \right|^{N+\sigma p}}\,\dxy \right) \no \\
     & = A(N,\frac{p}{2}) \sum_{i=1}^{N} \norm{V_{i} - V_{n,i}}_{L^p(\R^N \times \R^N)}^p \ra 0, \text{ as } n \ra \infty.
\end{align*}
This completes the proof.
\end{proof}
\subsection*{Completions} From Proposition \ref{prop3.4}, the seminorm $\left[\cdot\right]_{s,p,a}$ becomes a norm in the space $\cc(\R^N)$. However, $\cc(\R^N)$ is not complete with respect to this norm. This leads us to define the completion of $\cc(\R^N)$ with respect to $\left[\cdot\right]_{s,p,a}$. Let $V$ be the vector space of Cauchy sequences in $\cc(\R^N)$ and $V_0$ be the subspace of $V$ consisting of the sequences converging to zero. We define an equivalence relation on $V$ by
\begin{align*}
    \boldsymbol{u} \sim \boldsymbol{v} \; \iff \; \boldsymbol{u}-\boldsymbol{v} \in V_0,
\end{align*}
i.e., two sequences $\boldsymbol{u}=\{u_n\}$ and $\boldsymbol{v}=\{v_n\}$ in $V$ are said to be equivalent if and only if
\begin{align}\label{equivalence relation}
    \lim_{n \ra \infty}\left[u_n - v_n\right]_{s,p,a}=0.
\end{align}
The equivalence class containing a vector $\boldsymbol{u} \in V$ is called a coset, and it is denoted as $\boldsymbol{u}+V_0$ and defined by
\begin{align*}
    \boldsymbol{u}+V_0 :=\left\{\boldsymbol{u}+\boldsymbol{w} : \boldsymbol{w}\in V_0\right\}.
\end{align*}
The set of all cosets forms a quotient space and is defined as
\begin{align*}
    V/V_0:= \left\{\boldsymbol{u}+V_0: \boldsymbol{u}\in V\right\},
\end{align*}
and it becomes a vector space under the following addition and scalar multiplication:
\begin{align*}
    (\boldsymbol{u}+V_0) + (\boldsymbol{v}+V_0) = (\boldsymbol{u}+\boldsymbol{v}) +V_0  \text{ and } \alpha (\boldsymbol{u}+V_0)=\alpha \boldsymbol{u} +V_0, \; \forall \, \boldsymbol{u},\boldsymbol{v} \in V, \alpha \in \R.
\end{align*}
We define 
\begin{align}\label{quotient norm}
    \norm{\boldsymbol{u}+V_0}:= \lim_{n\ra \infty} \left[u_n\right]_{s,p,a}.
\end{align}
Since $\boldsymbol{u} \in V$, the sequence $\boldsymbol{u}= \{u_n\}$ is Cauchy with respect to $ \left[\cdot\right]_{s,p,a}$ and therefore the sequence $\{\left[u_n\right]_{s,p,a}\}$ is Cauchy in $\R$ by the following inequality
\begin{align*}
    \abs{\left[u_n\right]_{s,p,a} - \left[u_m\right]_{s,p,a}} \leq \left[u_n -u_m\right]_{s,p,a}.
\end{align*}
By completeness of $\R$, the sequence $\{\left[u_n\right]_{s,p,a}\}$ converges in $\R$ and hence the right hand quantity in \eqref{quotient norm} is well defined. Further, if $\boldsymbol{u}= \{u_n\}$ and $\boldsymbol{v}= \{v_n\}$ are Cauchy sequences such that $\{u_n-v_n\} \in V_0$, then by the following inequalities
\begin{align*}
    \left[u_n\right]_{s,p,a} \leq \left[u_n-v_n\right]_{s,p,a} + \left[v_n\right]_{s,p,a} \text{ and } \left[v_n\right]_{s,p,a} \leq \left[u_n-v_n\right]_{s,p,a} + \left[u_n\right]_{s,p,a},
\end{align*}
it follows that $\lim\limits_{n \ra \infty}\left[u_n\right]_{s,p,a} = \lim\limits_{n \ra \infty}\left[v_n\right]_{s,p,a}$. Thus, \eqref{quotient norm} is well defined, and it is independent of the choice of a representative of the coset.

\begin{remark}
   Using \cite[Section 10, Page 56]{Yosida_Functional}, we can verify that $\norm{\cdot}$ becomes a norm on $V/V_0$ and $V/V_0$ forms a Banach space with respect to $\norm{\cdot}$ and this space is unique upto isometric isomorphism. 
\end{remark}
Since $\Dsp$ is a completion of $\left(\cc(\R^N),[\cdot]_{s,p,a}\right)$, we can identify the spaces $V/V_0$ and $\Dsp$. Each element of $\Dsp$ is considered as an equivalence class $U=\langle\{u_n\}\rangle$, where the Cauchy sequence $\{u_n\}$ in $\cc(\R^N)$ is a representative of the class. Due to the isometric isomorphism between $\Dsp$ and $V/V_0$, the norm on $\Dsp$ in terms of its representative is given by 
\begin{align*}
    \norm{U}_{\Dsp}= \lim_{n \ra \infty} \left[u_n\right]_{s,p,a}.
\end{align*}

\noi \textbf{Proof of Theorem \ref{Theorem1.1}-(iii):} Using Proposition \ref{prop3.4}, it follows that $[\cdot]_{s,p,a}$ defines a norm on $\Wsp$. Now, we will show the existence of an isometric isomorphism between $\Dsp$ and $\Wsp$, thereby establishing that $\Wsp$ is a Banach space. Let $U= \langle \{u_n\} \rangle \in \Dsp$. By Proposition \ref{Prop4.1}, $u_n \ra u$ with respect to $\left[\cdot\right]_{s,p,a}$ for some $u \in \Wsp$. We define a map $\mathcal{J}: \Dsp \rightarrow \Wsp$ as 
    \begin{align*}
        \mathcal{J}(U) = u.
    \end{align*}
If $\{v_n\}$  is another representative of the equivalence class $U$, then using the equivalence relation \eqref{equivalence relation}, we have
\begin{align*}
    \lim_{n \ra \infty}\left[u-v_n\right]_{s,p,a} \leq  \lim_{n \ra \infty}\left[u -u_n\right]_{s,p,a} +  \lim_{n \ra \infty}\left[u_n -v_n\right]_{s,p,a} =0.
\end{align*}
Thus, the map $\mathcal{J}$ is well-defined, and one can verify that it is linear. Further,
the map $\mathcal{J}$ is an isometry by Proposition \ref{Prop4.1} i.e.,
\begin{align*}
    \norm{U}_{\Dsp} = \lim_{n \ra \infty}\left[u_n\right]_{s,p,a}= \left[u\right]_{s,p,a}.
\end{align*}
Now we claim that $\mathcal{J}$ is surjective. It follows from Proposition \ref{Prop3.2} that for every $w \in \Wsp$, there exists $\{w_n\} \subset \cc(\R^N)$ such that 
\begin{align*}
    \lim_{n \ra \infty}\left[w-w_n\right]_{s,p,a}=0.
\end{align*}
Since $\{w_n\}$ is Cauchy with respect to $\left[\cdot\right]_{s,p,a}$, it is a representative of some equivalence class $W$ i.e., $W= \langle \{w_n\}\rangle$. Thus, we obtain 
\begin{align*}
    \mathcal{J}(W)=w,
\end{align*}
which implies $\mathcal{J}$ is surjective. Hence $\mathcal{J}$ is a linear surjective isometry. \qed

\section{Fractional weighted Rellich inequality and finer embeddings}

First, we recall the Caffarelli-Kohn-Nirenberg inequality\cite{Caffarelli_kohn_1984}. \begin{theorem}[Caffarelli-Kohn-Nirenberg]\label{C-K-N-Theorem}
    Let $p,q,r,\alpha,\beta,\gamma$ and $l$ be real numbers such that $p,q \geq 1, r>0, 0\leq l \leq 1$, and
    \begin{align*}
        \frac{1}{p}+\frac{\alpha}{N},\frac{1}{q}+\frac{\beta}{N},\frac{1}{r}+\frac{m}{N}>0,
    \end{align*}
    where $m = l \gamma + (1-l)\beta$. Then there exists a positive constant $C$ such that 
   \begin{align*}
     \left(\int_{\R^N} |x|^{mr}|u(x)|^r\,\dx\right)^{\frac{1}{r}}   \leq C  \left(\int_{\R^N} |x|^{\alpha p}|\Gr u(x)|^p\,\dx\right)^{\frac{l}{p}} \left(\int_{\R^N} |x|^{\beta q}|u(x)|^q\,\dx\right)^{\frac{1-l}{q}}, \; \forall \, u \in \cc(\R^N),
   \end{align*} 
 if and only if the following relations hold
 \begin{align*}
     \frac{1}{r}+\frac{m}{N}= l \left(\frac{1}{p}+\frac{\alpha-1}{N} \right)+ (1-l) \left( \frac{1}{q}+\frac{\beta}{N}\right),
 \end{align*}
 with
 \begin{align*}
     0 \leq \alpha - \gamma  \text{ if } l >0,
 \end{align*}
 and 
 \begin{align*}
     \alpha -\gamma \le 1 \text{ if } l>0 \text{ and }\frac{1}{r}+\frac{m}{N}= \frac{1}{p}+\frac{\alpha-1}{N}.
 \end{align*}
\end{theorem}

For $s \in (1,2), \sigma = s-1, p\in (1,\frac{N}{s})$ and $a \in [0,\frac{N-sp}{2})$, we consider
 \begin{align}\label{choice}
     r=p, l=1, \alpha = -\frac{\sigma p+2a}{p}, \text{ and } \gamma=-\frac{sp + 2a}{p}.
 \end{align}
 Then, we notice that $m = \gamma=\alpha-1$ and 
$\frac{1}{r}+\frac{m}{N} =\frac{1}{p}-\frac{s p+2a}{Np} =     \frac{1}{p}+\frac{\alpha-1}{N}$. Moreover,
\begin{align*}
    \frac{1}{p}+\frac{\alpha}{N}= \frac{N-\sigma p-2a}{Np}>0 \text{ and } \frac{1}{r}+\frac{m}{N}=\frac{N-sp-2a}{Np}>0.
\end{align*}
Hence the assumptions of Theorem \ref{C-K-N-Theorem} are verified, and we get 
\begin{align}\label{C-K-N_1}
    C(N,p,\sigma,a) \int_{\R^N}\frac{\abs{u(x)}^p}{\abs{x}^{s p +2a}} \, \dx \le \int_{\R^N}\frac{\abs{\Gr u(x)}^p}{\abs{x}^{\sigma p +2a}} \, \dx.
\end{align}
Moreover, by \cite[Theorem 2.7]{Caffarelli_Kohn_2017},
\begin{align}\label{weighted Hardy}
 C(N,p,\sigma,a) \int_{\R^N}\frac{\abs{u(x)}^p}{\abs{x}^{\sigma p +2a}}\,\dx \le \iint_{\R^N \times \R^N} \frac{\abs{u(x)-u(y)}^p}{\abs{x-y}^{N + \sigma p}}\, \dxy, \; \forall \, u \in \cc(\R^N). 
\end{align}
Now observe that
\begin{align*}
    &\iint_{\R^N \times \R^N} \frac{\abs{\Gr u(x)-\Gr u(y)}^p}{\abs{x-y}^{N + \sigma p}}\, \dxy \\
    &\geq B(N,\frac{p}{2})\left(\sum_{i=1}^{N} \iint_{\R^N \times \R^N} \frac{\abs{\frac{\del u}{\del x_i}(x)-\frac{\del u}{\del x_i}(y)}^p}{\abs{x-y}^{N + \sigma p}}\, \dxy \right) \text{ (using Lemma \ref{lemma_1}-(ii))}\\
    & \geq B(N,\frac{p}{2})C \left(\sum_{i=1}^{N}\int_{\R^N}\frac{\abs{\frac{\del u}{\del x_i}(x)}^p}{\abs{x}^{\sigma p +2a}}\,\dx \right) \text{ (using \eqref{weighted Hardy} with $\frac{\del u}{\del x_i}$)}\\
    & \geq \frac{B(N,\frac{p}{2})}{A(N,\frac{p}{2})} C\int_{\R^N}\frac{\abs{\Gr u(x)}^p}{\abs{x}^{\sigma p +2a}}\,\dx \text{ (using Lemma \ref{lemma_1}-(i))}\\
    & \geq \frac{B(N,\frac{p}{2})}{A(N,\frac{p}{2})} C \int_{\R^N}\frac{\abs{u(x)}^p}{\abs{x}^{s p +2a}} \, \dx,   \text{ (using \eqref{C-K-N_1})}.
\end{align*}
Therefore, we have 
\begin{proposition}\label{weighted higher}
Let $N \ge 2, s \in (1,2), \sigma = s-1, p \in (1, \frac{N}{s})$, and $a \in [0, \frac{N-sp}{2})$. Then 
\begin{align}\label{weighted Hardy higher}
    \iint_{\R^N \times \R^N} \frac{\abs{\Gr u(x)-\Gr u(y)}^p}{\abs{x-y}^{N + \sigma p}}\, \dxy \geq C\int_{\R^N}\frac{\abs{u(x)}^p}{\abs{x}^{s p +2a}} \, \dx, \;\forall \, u \in \cc(\R^N),
\end{align}
where $C=C(N,p,\sigma,a)$. 
\end{proposition}

Next, using \eqref{weighted Hardy higher}, we study the finer embeddings of $\Dsp$. Let $\Omega$ be an open set in $\R^N$. For $f \in M(\Omega)$ we define the one dimensional decreasing rearrangement $f^*$ of $f$ as $f^*(\tau) = \inf \{t>0 : \mu_{f}(t) < \tau\}$ for $\tau > 0$. Then, the following identity holds (see \cite{EdEv}) 
\begin{align}\label{identity-1}
    \int_{\Omega} f(x) \, \dx = \int_{0}^{\abs{\Omega}} f^*(t) \, \dt. 
\end{align}

\noi \textbf{Proof of Theorem \ref{Theorem1.1}-(iv):} 
For $u \in \cc(\R^N)$, consider the set $A_t = \{ x \in \R^N : \abs{u(x)}^p > t\}$ and we write 
    \begin{align*}
        \abs{u(x)}^p = \int_{0}^{\abs{u(x)}^p} \dt = \int_{0}^{\infty} \chi_{t}(x) \, \dt, \text{ where } \chi_{t}(x) = 1, \text{ if } x \in A_t, \text{ and } \chi_{t}(x) = 0, \text{ if } x \in A_t^c.
    \end{align*}
Using Fubini's theorem, we write
\begin{align*}
     \int_{\R^N}\frac{\abs{u(x)}^p}{\abs{x}^{sp +2a}} \, \dx =\int_{\R^N} \left( \int_{0}^{\infty}   \chi_{t}(x) \, \dt \right) \frac{\dx}{\abs{x}^{sp +2a}} =  \int_0^{\infty} \left( \int_{\R^N} \frac{\chi_{t}(x)}{\abs{x}^{sp +2a}} \, \dx \right) \,\dt.
\end{align*}
Denote $f(x) = \abs{x}^{-(sp +2a)}$ for $x \in \R^N$. Using \eqref{identity-1},  
\begin{align*}
    \int_{\R^N} f(x) \chi_{t}(x) \, \dx = \int_{A_t} f(x) \, \dx = \int_{0}^{|A_t|} f^*(t) \, \dt. 
\end{align*}
By \cite[Example 2.1.2]{biswas2021}, we see that $f^*(t) = \left( \frac{\omega_N}{t}\right)^{\frac{sp+2a}{N}}$, were $\omega_N$ is the measure of a unit ball in $\R^N$. This implies
\begin{align*}
    \int_{\R^N} f(x) \chi_{t}(x) \, \dx = \omega_N^{\frac{sp+2a}{N}} \int_{0}^{|A_t|} \frac{\dt}{t^{\frac{sp+2a}{N}}} = \frac{N \omega_N^{\frac{sp+2a}{N}}}{N-sp-2a} \abs{A_t}^{\frac{N-sp-2a}{N}}. 
\end{align*}
Observe that $\abs{A_t} = \mu_{\abs{u}^p}(t)$ for $t>0$. Therefore, 
\begin{align*}
    \int_{\R^N}\frac{\abs{u(x)}^p}{\abs{x}^{sp +2a}} \, \dx = \frac{N \omega_N^{\frac{sp+2a}{N}}}{N-sp-2a} \int_{0}^{\infty} \mu_{\abs{u}^p}(t)^{\frac{p}{p^{\star}}} \, \dt = \omega_N^{\frac{sp+2a}{N}} |\abs{u}^p|_{\frac{p^{\star}}{p}, 1} = \omega_N^{\frac{sp+2a}{N}} |u|^p_{p^{\star}, p}.
\end{align*}
Hence using \eqref{weighted Hardy higher} we obtain 
\begin{align}\label{embedding}
    |u|_{p^{\star}, p} \le C(N,s,p, a) [u]_{s,p,a}, \; \forall \, u \in \cc(\R^N).
\end{align}
Thus, \eqref{embedding} with the definition of $\Dsp$ yields \eqref{finer embedding}.  \qed 

In the following proposition, using a similar layer cake decomposition technique, we show that $L^{p^\star,p}(\R^N)$ is strictly contained in $\Lpastar$. 

\begin{proposition}\label{finer}
   Let $N \ge 2, s \in (1,2), \sigma = s-1, p \in (1, \frac{N}{s})$, and $a \in [0, \frac{N-sp}{2})$. Then $L^{p^\star,p}(\R^N) \subsetneq \Lpastar$, where $p^\star = \frac{Np}{N-sp-2a}.$
\end{proposition}
\begin{proof}
   For $t>0$, we consider the set $A_t = \{ x \in \R^N : \abs{u(x)}^{\ps} > t\}$, and write 
    \begin{align*}
        \abs{u(x)}^{\ps}= \int_{0}^{\abs{u(x)}^{\ps}} \dt = \int_{0}^{\infty} \chi_{t}(x) \, \dt, \text{ where } \chi_{t}(x) = 1, \text{ if } x \in A_t, \text{ and } \chi_{t}(x) = 0, \text{ if } x \in A_t^c.
    \end{align*}
Using Fubini's theorem, 
\begin{align*}
     \int_{\R^N}\frac{\abs{u(x)}^{\ps}}{\abs{x}^{\frac{2a \ps}{p}}} \, \dx =\int_{\R^N} \left( \int_{0}^{\infty}   \chi_{t}(x) \, \dt \right) \frac{\dx}{\abs{x}^{\frac{2a \ps}{p}}} =  \int_0^{\infty} \left( \int_{\R^N} \frac{\chi_{t}(x)}{\abs{x}^{\frac{2a \ps}{p}}} \, \dx \right) \,\dt.
\end{align*}
Denote $f(x) = \abs{x}^{-\frac{2a \ps}{p}}$ for $x \in \R^N$. Using \eqref{identity-1},  
\begin{align*}
    \int_{\R^N} f(x) \chi_{t}(x) \, \dx = \int_{A_t} f(x) \, \dx = \int_{0}^{|A_t|} f^*(t) \, \dt. 
\end{align*}
By \cite[Example 2.1.2]{biswas2021}, $f^*(t) = \left( \frac{\omega_N}{t}\right)^{\frac{2a}{N-sp}}$ and hence
\begin{align*}
    \int_{\R^N} f(x) \chi_{t}(x) \, \dx = \omega_N^{\frac{2a}{N-sp}} \int_{0}^{|A_t|} \frac{\dt}{t^{\frac{2a}{N-sp}}} = \frac{ \omega_N^{\frac{2a}{N-sp}} (N-sp)}{N-sp-2a} \abs{A_t}^{\frac{N-sp-2a}{N-sp}}, 
\end{align*}
where $\abs{A_t} = \mu_{\abs{u}^{\ps}}(t)$ for $t>0$. Therefore, 
\begin{align*}
    \int_{\R^N}\frac{\abs{u(x)}^{\ps}}{\abs{x}^{\frac{2a \ps}{p}}} \, \dx = \frac{ \omega_N^{\frac{2a}{N-sp}}(N-sp)}{N-sp-2a} \int_{0}^{\infty} \mu_{\abs{u}^{\ps}}(t)^{\frac{\ps}{p^{\star}}} \, \dt = \omega_N^{\frac{2a}{N-sp}} |\abs{u}^{\ps}|_{\frac{p^{\star}}{\ps}, 1} = \omega_N^{\frac{2a}{N-sp}} |u|^{\ps}_{p^{\star}, \ps}.
\end{align*}
We know that $L^{q,r}(\R^N) \subsetneq L^{q,s}(\R^N)$ for $r<s$. Taking $q=p^\star, s=\ps$ and $r=p$ yields 
\begin{align}\label{inclusion}
    L^{p^\star,p}(\R^N) \subsetneq L^{p^\star,\ps}(\R^N) = \Lpastar.
\end{align}
Hence the strict inclusion holds.
\end{proof}

\begin{remark}
    From Remark \ref{subset} and Theorem \ref{Theorem1.1}-(iv), we have 
    \begin{align*}
        \mathring{W}^{s,p}_a(\R^N) \subset L^{p^\star,p}(\R^N), \text{ for } s \in (1,2) \text{ and } p \in (1, \frac{N}{s}).
    \end{align*}
\end{remark}

\noi \textbf{Acknowledgments.} 
The first author acknowledges the Science and Engineering Research Board, Government of India, for national postdoctoral fellowship. The second author acknowledges the support of the CSIR fellowship, file no. 09/1125(0016)/2020--EMR--I.

\bibliographystyle{abbrv}

% \Addresses	 
		
\end{document}